\documentclass{amsart}

\usepackage{amsmath}\usepackage{amssymb}
\usepackage{amsfonts}
\usepackage[mathscr]{eucal}
\usepackage{afterpage}
\usepackage[draft]{hyperref}
\usepackage{graphicx}
\usepackage{pstricks}

\numberwithin{equation}{section}

\newtheorem{theorem}{Theorem}[section]

\newtheorem{lemma}{Lemma}[section]

\newtheorem{proposition}{Proposition}[section]

\newcommand{\be}{\begin{equation}}
\newcommand{\ee}{\end{equation}}
\newcommand{\ba}{\begin{array}}
\newcommand{\ea}{\end{array}}
\newcommand{\beas}{\begin{eqnarray*}}
\newcommand{\eeas}{\end{eqnarray*}}
\newcommand{\bea}{\begin{eqnarray}}
\newcommand{\eea}{\end{eqnarray}}

\newcommand{\lb}{\label}

\newcommand{\al}{\alpha}

\newcommand{\R}{\ensuremath{\mathbb R}}
\newcommand{\C}{\ensuremath{\mathbb C}}
\newcommand{\T}{\ensuremath{\mathbb T}}

\newcommand{\Z}{{\mathbf Z}}
\newcommand{\W}{{\mathbf W}}
\newcommand{\proj}{{\mathbf \Pi}}

\newcommand{\PP}{{\mathbf P}}
\newcommand{\dproj}{{\mathbf L}}
\newcommand{\inprod}[2]{\left\langle{#1},{#2}\right\rangle}

\newcommand{\A}{\boldsymbol{\mathrm{A}}  }
\newcommand{\B}{\boldsymbol{\mathrm{B}}  }

\newcommand{\CC}{\boldsymbol{\mathrm{C}}  }

\newcommand{\Vort}{\boldsymbol{\mathrm{Vort}} \, }
\newcommand{\Grad}{\boldsymbol{\mathrm{Grad}} \, }

\newcommand{\Curl}{\boldsymbol{\mathrm{Curl}} \, }
\newcommand{\Curln}{\boldsymbol{\mathrm{Curl}}_{\xh}  }

\newcommand{\DDelta}{\boldsymbol{\mathrm{\Delta}} \, }
\newcommand{\nnabla}{\boldsymbol{\mathrm{\nabla}} \, }
\newcommand{\Div}{\mathrm{Div} \, }
\newcommand{\p}{\boldsymbol{p}}
\newcommand{\calD}{\mathcal D}

\newcommand{\calP}{\mathcal P}
\newcommand{\calO}{\mathcal O}

\newcommand{\ve}{{\bf e}}
\newcommand{\vomega}{\boldsymbol{\mathrm{\omega}} \, }

\newcommand{\vu}{{\bf u}}
\newcommand{\vU}{{\bf U}}
\newcommand{\vf}{{\bf f}}

\newcommand{\vv}{{\bf v}}
\newcommand{\vw}{{\bf w}}

\newcommand{\vz}{{\overline{\bf z}}}
\newcommand{\ovz}{{\bf z}}
\newcommand{\x}{{\bf x}}

\newcommand{\hxi}{\widehat{\boldsymbol{\xi}}}
\newcommand{\xh}{\widehat{\x}}

\newcommand{\vp}{{\bf p}}

\newcommand{\tvu}{{\tilde{\vu}}}
\newcommand{\tvuN}{{\tilde{\vu}_N}}
\newcommand{\tvv}{{\tilde{\vv}}}
\newcommand{\tvw}{{\tilde{\vw}}}

\renewcommand{\tilde}{\widetilde}
\newcommand{\Aphalfe}{\A^{s+1/2} e^{\psi(r \cos\theta) \A^{1/2}}}
\newcommand{\Apone}{\A^{s+1} e^{\psi(r \cos\theta) \A^{1/2}}}
\newcommand{\Ape}{\A^{s} e^{\psi(r \cos\theta) \A^{1/2}}}

\newcommand{\Aphalfea}{\A^{s+1/2} e^{\psi \A^{1/2}}}

\newcommand{\Apea}{\A^{s} e^{\psi \A^{1/2}}}

\author{M. Ganesh}
\address{Department of Mathematical and Computer Sciences, 
        Colorado School of Mines, Golden, CO 80401}
\email{mganesh@mines.edu}	
\author{Q. T. Le Gia}
\address{School of Mathematics and Statistics, 
         University of New South Wales, Sydney, NSW 2052, Australia}
\email{qlegia@unsw.edu.au}
\author{I. H. Sloan}
\address{School of Mathematics and Statistics, 
         University of New South Wales, Sydney, NSW 2052, Australia}
\email{i.sloan@unsw.edu.au}

\title[A pseudospectral method for NSE\lowercase{s} on rotating spheres]
{A pseudospectral quadrature method for Navier-Stokes 
equations on  rotating spheres}
\date{\today}
\begin{document}
\begin{abstract}
In this work, we describe, analyze, and implement a pseudospectral 
quadrature method  for a global computer modeling of the incompressible 
surface Navier-Stokes equations on the rotating  unit sphere. 
Our spectrally accurate numerical error analysis is based on the Gevrey 
regularity of 
the solutions of the  Navier-Stokes equations on the sphere. The scheme 
is designed for convenient application  of fast evaluation techniques 
such as the fast Fourier transform (FFT), and the implementation is based 
on a stable adaptive time discretization. 
\end{abstract}
\maketitle
\setcounter{equation}{0}
\section{Introduction}\label{introduction}
In this paper we develop a pseudospectral quadrature method for the 
surface Navier-Stokes partial differential equations (PDEs) 
on the rotating unit sphere.
Whereas the finite element method is best
suited for handling non-smooth processes, 
the spectral global basis computer models are very efficient and 
perform extremely well for processes with smooth regularity.
For example,  exponential convergence properties of the 
global Fourier basis  spectral Galerkin methods (without quadrature)
for  the Ginzburg-Landau and Navier-Stokes PDEs on two dimensional  
periodic cells are based on the Gevrey regularity of solutions of 
the PDEs~\cite{doelman_titi, jones_titi}.

The complex three dimensional flows in the atmosphere and oceans are 
considered 
to be accurately modeled by the Navier-Stokes PDEs of fluid mechanics
together with classical thermodynamics~\cite{atmosp_model_books}. 
Difficulties in computer modeling in these PDEs resulted
in several simplified models for which spectral
approximations are well known~\cite{Boyd, spect_book,atmosp_model_books}.
A famous open problem is to prove the global regularity for the three
dimensional incompressible Navier-Stokes PDEs~\cite{tao_blog}. 
However, the  precise Gevrey regularity of the unique solution of the 
(practically relevant)  surface Navier-Stokes PDEs on the rotating sphere 
was proved in~\cite{cao_rammaha_titi99}.
(Because the Earth's surface is an approximate sphere, a standard
surface model, to study global atmospheric circulation on large 
planets,  is the sphere.)

Consequently, a natural next step is to describe,  
analyze, and implement an exponentially converging 
pseudospectral method for the Navier-Stokes  PDEs on the rotating sphere.
In addition to the continuous model regularity results in~\cite{cao_rammaha_titi99}, 
this paper is  also motivated by the recent work~\cite{fengler_freeden},
where  discrete computer modeling  of the Navier-Stokes PDEs on one-dimensional 
and toroidal domains~\cite{debussche_dubois_temam} 
was extended to the unit sphere.

This paper is concerned with both implementation of our  algorithm and its numerical
analysis. The main  numerical analysis contributions 
compared to  results in~\cite{cao_rammaha_titi99}, for the continuous
problem,  and in~\cite{fengler_freeden}, for a discrete problem, 
are as follows.

For the spatially discrete pseudospectral quadrature 
Galerkin solutions of the Navier-Stokes equations, 
we prove (i)  the stability (that is, uniform boundedness of  
approximate solutions, independent of their truncation parameter $N$),  
see Theorem~\ref{stability_thm}; 
and (ii) a  spectrally accurate 
rate of convergence  [that is, $\calO(N^{-2s})$ accuracy, with $s$ 
depending on the smoothness of input data], see  Theorem~\ref{conv_anal}.
We achieve these results by first generalizing the main regularity result 
in~\cite{cao_rammaha_titi99}  to complex valued times 
(see Theorem~\ref{complex_est} and \ref{complex_estT}), and then using this
to prove the spectral rate of convergence of the time derivative of
a Stokes projection  comparison function (see Theorem~\ref{der_cont_est}). 
This time-derivative error result plays a crucial role in the proof of  
spectral  convergence of the approximate solutions 
(see the proof of Theorem~\ref{conv_anal}). 
We note that  the main result in~\cite[page~978]{fengler_freeden} establishes only 
convergence of the semi-discrete  Galerkin method  without quadrature 
in $L^p$ norms, but does not establish either stability or  
rate of convergence of the scheme.

The rate of  convergence results, supported by numerical experiments, 
formed the core part of research on 
the   Navier-Stokes equations on two dimensional  domains 
over the last few decades, 
see~\cite{debussche_dubois_temam,Foias_book, temam79} and references therein.
There is a vast literature on numerical methods 
and analysis for the Navier-Stokes PDE
on bounded Euclidean domains~(see~\cite{debussche_dubois_temam, Foias_book, temam79} 
and references therein), 
but their counterparts on closed manifolds 
are rarer (see~\cite{fengler_freeden} and references therein).
The implementation of the scheme in~\cite[page~978]{fengler_freeden} is based 
on a fixed time-step explicit Runge-Kutta method that has 
a small stability region for the systems of  ordinary differential equations 
arising from the spatial discretization.

The outline of this paper is as follows. In the next section, we recall
various known preliminary results associated with the Navier-Stokes
PDEs on the unit sphere, in strong and weak form. In Section 3, we introduce
essential  computational and numerical analysis tools required for the 
discretization and analysis of the Navier-Stokes equations. In Section 4,
we describe and prove spectral accuracy of  a pseudospectral quadrature method 
and give implementation details required to apply the FFT 
and adaptive-in-time simulation  of the Navier-Stokes equations.
In Section 5, we demonstrate computationally the accuracy and applicability  
of the algorithm for well known benchmark examples.

\section{Navier-Stokes equations on the rotating unit sphere}
The surface Navier-Stokes equations (NSE) describing a tangential, incompressible
atmospheric stream on  the rotating two-dimensional unit sphere 
$S \subset \R^3$ can be written as~\cite{cao_rammaha_titi99, ilin91, ilin94, ilin_filatov,temam_wang}
\be\label{NSE_sphere}
   \frac{\partial}{\partial t} \vu + \nnabla_\vu \vu - \nu \DDelta \vu 
+ \vomega \times \vu 
+ \frac{1}{\rho} \Grad p = \vf,
   \qquad \Div \vu =0, \qquad \vu|_{t=0} = \vu_0 \qquad \text{on}~~S.
\ee
Here $\vu = \vu(\xh,t) = \left(u_1(\xh,t),u_2(\xh,t), u_3(\xh,t)\right)^T$ 
is the unknown {\em tangential} divergence-free velocity field at $\xh \in S$ 
and $t \in [0,T]$, $p = p(\xh,t)$ is the unknown pressure. The known
components in (\ref{NSE_sphere}) are the constant viscosity and density of the fluid, respectively denoted by $\nu, \rho$, the normal vector field
$\vomega = \vomega(\xh)=\omega(\xh)\xh$ for the Coriolis acceleration term,
and the external flow driving vector field $\vf = \vf(\xh,t)$. The
Coriolis function $\omega$ is given by $\omega(\xh)=2\Omega\cos\theta$,
where $\Omega$ is the angular velocity of the rotating sphere, and $\theta$
is the angle between $\xh$ and the north pole.
The vorticity of the flow associated with the NSE (\ref{NSE_sphere}), 
in the curvilinear coordinate system,  is a normal vector field, defined,  for
a fixed $t \geq 0$,  by 
\begin{equation}\label{eq:vort}
 \Vort \vu (\xh,t)  = \Curln \vu (\xh,t) = \xh \Delta \Psi(\xh,t), \qquad
 \qquad \xh \in S,
\end{equation} 
for some scalar-valued  vorticity stream function $\Psi$.

All spatial derivative operators in  (\ref{NSE_sphere})-(\ref{eq:vort}) are  
 surface differential operators, obtained by restricting the
corresponding domain operators (defined in a  neighborhood of $S$)
to the unit sphere, using standard differential geometry concepts on closed
manifolds in $\R^3$~\cite{ilin91, ilin94}.
Using the fact that the outward unit normal at $\xh\in S$ is $\xh$, the 
$\Curl$ of a scalar function $v$, of a normal vector field $\vw =  w \xh$, 
and of a tangential vector field $\vv$ on $S$ are 
respectively defined by
\begin{equation}\label{Curl}
\Curl v = -\xh \times \Grad v,  \qquad \Curl \vw = -\xh \times \Grad w, \qquad \Curln \vv = -\xh~ \Div(\xh \times \vv).
\end{equation}
The surface diffusion operator acting on  tangential  
vector fields  on $S$ is  denoted by $\DDelta$ 
(known as the Laplace-Beltrami or Laplace-de Rham operator) 
and  is defined as 
\begin{equation}\label{deRham}
 \DDelta \vv = \Grad \Div \vv - \Curl \Curln \vv.
\end{equation}

The following relations connecting the above operators will be used throughout
the paper:
\begin{equation} \label{curln_curl_psi}
   \Div \Curl v = 0, \qquad
   \Curln \Curl v = -\xh \Delta v, \qquad 
 \DDelta \Curl v = \Curl \Delta v,
\end{equation}
\begin{equation}\label{covar_long}
2 \nnabla_\vw \vv = -\Curl(\vw\times \vv)+\Grad(\vw\cdot \vv)-\vv\Div \vw+
   \vw \Div \vv - 
   \vv\times\Curln \vw - \vw \times \Curln \vv.
\end{equation}
In particular, for tangential divergence-free vector fields, such as the 
solution $\vu$ of the NSE, using (\ref{covar_long}), the nonlinear term in (\ref{NSE_sphere}) can be written as 
\begin{equation}\label{short_nonlin1}
\nnabla_\vu \vu = \Grad \frac{|\vu|^2}{2} - \vu \times \Curln \vu.
\end{equation}
   
\subsection{A weak formulation}
A standard technique for removing the scalar pressure field from
the Navier-Stokes equations is to multiply the first equation
in (\ref{NSE_sphere}) by test functions $\vv$ from a space with elements
having  properties of the unknown velocity field $\vu$ (in particular,
$\Div \vv =0$) and then integrate to obtain a weak formulation. 
(The unknown $p$, can then be computed by
solving a pressure Poisson equation, obtained by applying 
the surface divergence operator in (\ref{NSE_sphere}).)

To this end, we introduce the standard inner products on 
the space of all square integrable (i) scalar functions 
on $S$, denoted by $L^2(S)$; and (ii) tangential vector
fields on $S$, denoted by $L^2(TS)$:
\bea
(v_1,\;v_2) &=& (v_1,\;v_2)_{L^2(S)} ~~= \int_{S} v_1 \overline{v_2}~dS, 
\quad \qquad v_2, v_2 \in L^2(S), \label{scal_ip} \\ 
(\vv_1,\;\vv_2)   &=& (\vv_1,\;\vv_2)_{L^2(TS)} = \int_{S} \vv_1 \cdot 
\overline{\vv_2}~dS, \qquad  \vu,\vv \in L^2(TS), \label{vec_ip}
\eea
where $dS = \sin \theta d\theta d\phi$.
Throughout the paper, the  induced norm on  $L^2(TS)$ is   
denoted by $\| \cdot \|$ and for other inner product spaces, 
say $X$ with inner product $(\cdot,\;\cdot )_{X}$, the associated norm is
denoted by  $\| \cdot \|_X$.  For example, for $s > 0$, 
standard norms in the scalar and vector valued functions
Sobolev spaces $H^s(S)$ and  $H^s(TS)$ are  denoted 
by $\| \cdot \|_{H^s(S)}$  and $\| \cdot \|_{H^s(TS)}$,
respectively. Since $H^0(TS) = L^2(TS)$, $\| \cdot \|_{H^0(TS)} = \| \cdot \|$.

We have the following identities for appropriate scalar and
vector fields~\cite[(2.4)-(2.6)]{ilin91}:
\begin{eqnarray}
(\Grad \psi,\; \vv) = - (\psi,\; \Div \vv),  && \qquad 
(\Curl \psi,\;\vv)  = (\psi,\; \Curln \vv), 
    \label{ip_identities1} \\   
(\Curl \Curln \vw,\; \ovz) &=& (\Curln \vw,\; \Curln \ovz).
\label{ip_identities2}
\end{eqnarray}
In \eqref{ip_identities1}, the  $L^2(TS)$ inner
product is used on the left hand side and the $L^2(S)$ inner product 
is used on the right hand side. Throughout the paper, we identify a 
normal vector field $\vw$ with a
scalar field $w$ and hence 
\begin{equation}\label{normal_ip}
(\psi,\; \vw) := (\psi,\; w)_{L^2(S)}, \qquad 
\vw = \xh w, \qquad \psi, w \in L^2(S).
\end{equation}

Using (\ref{curln_curl_psi}), smooth ($C^\infty$) tangential fields on $S$
can be decomposed into two components, one in the space of all 
divergence-free fields and the other through the 
Hodge decomposition 
theorem~\cite{aubin}:
\begin{equation} \label{Hodge}
  C^\infty(TS) =   C^\infty(TS;\Grad) \oplus  C^\infty(TS;\Curl),
\end{equation}
where 
\begin{equation} \label{orth_space}
  C^\infty(TS;\Grad) = \{\Grad \psi: \psi \in C^\infty(S)\},\;
   C^\infty(TS;\Curl) = \{\Curl \psi: \psi \in C^\infty(S)\}.
\end{equation}
For $s \geq 0$,  let $C^{\infty,s}(TS;\Curl)$ denote the closure of
$C^\infty(TS;\Curl)$ in the   $H^s(TS)$ norm. In particular, following
\cite{cao_rammaha_titi99} we introduce a simpler notation 
\beas
  H &=& \mbox{ closure of } C^{\infty}(TS;\Curl) 
        \mbox{ in } L^2(TS)\; =\;  C^{\infty,0}(TS;\Curl), \\
  V &=& \mbox{ closure of } C^{\infty}(TS;\Curl) 
        \mbox{ in } H^1(TS)\; =\; C^{\infty,1}(TS;\Curl).
\eeas
Using the Gauss surface divergence theorem, for any 
scalar valued function $v$ on $S$ with
$\Grad v \in L^2(TS)$, using (\ref{ip_identities1}), 
we have
\begin{equation} \label{div_thm}
(\Grad v,\;\vw) =  \int_{S} \Grad v \cdot  \overline{\vw}~dS
=  - \int_{S}  v \cdot  \Div \overline{\vw}~dS = 0, 
\quad \vw \in  V,
\end{equation}
and hence the unknown pressure can be eliminated from the
first equation in (\ref{NSE_sphere}) through the weak formulation.

Following~\cite[Page~567]{ilin91}, for the diffusion part of the
NSE, we consider the Stokes operator
\begin{equation}\label{A}
\A = \Curl \Curln. 
\end{equation}
Using (\ref{deRham}) and (\ref{curln_curl_psi}), it is 
easy to see that the  Stokes operator is the restriction of 
the vector Laplace-de Rham operator $-\DDelta$
on $V$; $\A = - \PP_{\Curl} \DDelta$, 
where  $\PP_{\Curl}:~L^2(TS)~\rightarrow~H$ is the orthogonal projection 
onto the divergence-free tangent space.

For each positive 
 integer $L=1,2,\ldots$, the eigenvalue $\lambda_L$ and the corresponding 
eigenvectors of the   Stokes operator $\A$  are given by 
\begin{equation}\label{stok_eig}
\lambda_L=L(L+1),  \qquad     \Z_{L,m}(\theta,\varphi) = \lambda^{-1/2}_L \Curl Y_{L,m}(\theta,\varphi),
\quad m=-L, \ldots, L, 
\end{equation}
where $Y_{L,m}$ are the scalar orthonormal spherical harmonics of degree $L$,
defined by
\begin{equation}\label{sph_har}
   Y_{L,m}(\theta,\varphi) = 
        \left[\frac{(2L+1)}{4\pi}
      \frac{(L-|m|)!}{(L+|m|)!} \right]^{1/2} P^{m}_{L}(\cos \theta) e^{im\varphi}, \quad m=-L, \ldots, L, 
\end{equation}
with  $P^m_L$ being the associated Legendre polynomials 
so that $\overline{Y_{L,m}} = (-1)^m Y_{L,-m}$.

The spectral property $\A  \Z_{L,m} = \lambda_L  \Z_{L,m}$ follows from the fact that $Y_{L,m}$ are eigenfunctions of the scalar Laplace-Beltrami operator $-\Delta$ 
with eigenvalues $\lambda_L$, the definition of 
the Stokes operator $\A$ in (\ref{A}), $\Z_{L,m}$ in
(\ref{stok_eig}), (\ref{curln_curl_psi}), and (\ref{Curl}).
Since $\{Y_{L,m}: L=0,1,\ldots; m =-L,\ldots, L\}$ is an orthonormal
basis for $L^2(S)$, it is easy to see that 
$\{\Z_{L,m}: L=1,\ldots; m =-L,\ldots, L\}$ is an orthonormal basis for $H$.
Thus an arbitrary $\vv \in H$ can be written as
\begin{equation}\label{Fourier_coeff}
\vv=\sum_{L=1}^\infty \sum_{m=-L}^L \widehat{\vv}_{L,m} \Z_{L,m},
\qquad  
\widehat{\vv}_{L,m} = \int_S \vv \cdot \overline{\Z_{L,m}} dS =
(\vv,\;\Z_{L,m}).
\end{equation}

We consider a subset of $H$,
\begin{equation}\label{Gev0_space}
 \calD(\A^{s/2}) = 
\left\{\vv \in  H \; : \vv = \sum_{L=1}^\infty \sum_{m=-L}^{L} \widehat{\vv}_{L,m} \Z_{L,m},
              \quad\sum_{L=1}^\infty \sum_{m=-L}^L \lambda^s_L |\widehat{\vv}_{L,m}|^2<\infty 
               \right\},
\end{equation}
which is the divergence-free subset of the Sobolev space $H^s(TS)$.
For every $\vv \in \calD(\A^{s/2})$, 
we set 
\begin{equation}\label{Hs-norm}
\| \vv \|_{H^s(TS)} = \left[\sum_{L=1}^\infty \sum_{m=-L}^L  
\lambda^s_L  |\widehat{\vv}_{L,m}|^2 \right]^{1/2},
\end{equation}
and for $\vv\in \calD(\A^{s/2})$, we define
\begin{equation}\label{As-def}
\A^{s/2} \vv := 
\quad\sum_{L=1}^\infty 
  \sum_{m=-L}^L \lambda^{s/2}_L \widehat{\vv}_{L,m} \Z_{L,m} \quad \in H.
\end{equation}
For a tangential vector field $\vv$ on $S$, we define the Coriolis operator $\CC$,
\begin{equation}\label{C}
(\CC\vv)(\xh) = \vomega(\xh) \times \vv(\xh) = \omega(\xh) (\xh \times \vv), \qquad 
\omega(\xh)  = 2 \Omega \cos \theta. 
\end{equation}

To treat the nonlinear term in  (\ref{NSE_sphere}), 
we consider the trilinear form $b$ on  $V \times  V \times V$, 
defined as  
\begin{equation}\label{trilinear}
b(\vv,\vw,\ovz) = (\nnabla_\vv \vw,  \ovz) =  \int_{S} \nnabla_\vv \vw \cdot
\vz~dS, \qquad \vv, \vw, \ovz \in  V.
\end{equation}
Using (\ref{covar_long}) and (\ref{ip_identities1}), for divergence free
fields $\vv, \vw, \ovz$,  the trilinear
form can be written as
\be
b(\vv,\vw,\ovz) =  \frac{1}{2} \int_S \left[-\vv \times \vw \cdot \Curln \vz  +
             \Curln \vv \times \vw \cdot \vz - \vv \times \Curln \vw \cdot \vz\right]~dS.
\label{short_b}
\ee
Moreover~\cite[Lemma 2.1]{ilin91}
\be \lb{skew}
   b(\vv,\vw,\vw)=0, \qquad b(\vv,\ovz,\vw) = -b(\vv,\vw,\ovz)
\qquad \vv, \vw, \ovz \in V.
\ee
Throughout the paper,  the space $V$ is equipped 
with the norm $\| \cdot \|^2_V = \left(\A \cdot, \cdot\right)$.

Thus, using (\ref{deRham}),  (\ref{ip_identities1}), (\ref{A}), and (\ref{short_b}), a \emph{weak solution} of the Navier-Stokes equations (\ref{NSE_sphere}) is a vector field 
$\vu \in L^2([0,T];V)$ 
with $\vu(0)= \vu_0$ that satisfies the weak form 
\begin{equation} \label{weak_form}
  (\vu_t,\vv) + b(\vu,\vu,\vv) + \nu(\Curln \vu, \Curln \vv) 
+ (\CC\vu, \vv)  = (\vf,\vv), \qquad  \vv \in  V. 
\end{equation}
This weak formulation can be written in operator equation form on 
$V^*$, the adjoint of $V$:
Let $\vf \in L^2([0,T];V^*)$ and 
$\vu_0 \in H$. 
Find a vector field  $\vu \in L^2([0,T]; V)$, with
$\vu_t \in L^2([0,T];V^*)$
such that 
\begin{equation}\label{op_form}
    \vu_t + \nu \A\vu + \B(\vu,\vu) + \CC\vu = \vf, \qquad \vu(0) = \vu_0,
\end{equation}
where the bilinear form $\B(\vu,\vv) \in V^*$ is defined by 
\begin{equation}\label{B}
  (\B(\vu,\vv),\vw) = b(\vu,\vv,\vw) \qquad  \vw \in V.
\end{equation}
In the subsequent error analysis, we need the following 
estimate for the nonlinear term (see Lemma~\ref{lem:weaklip} 
in Appendix):
\be\label{Adelta}
  \|\A^{-\delta} \B(\vu,\vv) \| \le  \begin{cases}C\|\A^{1-\delta}\vu\| \|\vv\|\le C\|\A^{1/2}\vu\|\|\vv\|, &\cr
                                          C\|\vu\|\|\A^{1-\delta}\vv\| \le
                                          C\|\vu\|\|\A^{1/2}\vv\|, &
                                   \end{cases}    \quad   \delta \in
                                   (1/2,1) \quad \vu, \vv \in V.
\ee
In \eqref{Adelta}, as throughout the paper, $C$ is a generic constant
independent of $\vu$ and $\vv$, (and the discretization parameter 
$N$ introduced in Section~~\ref{sec:proj}).  
From \eqref{Adelta} we deduce the weak Lipschitz continuity
property
\begin{equation}
 \|\A^{-\delta}(\B(\vv,\vv)-\B(\vw,\vw))\| \le C \|\vv-\vw\|, \quad \delta \in (1/2,1), 
\quad \text{if } \|\A^{1/2}\vv\| ,\|\A^{1/2}\vw\| < C.
\label{weakLip}
\end{equation}

The existence and uniqueness of the  solution 
$\vu \in L^2([0,T]; V)$ of the weak formulation (\ref{weak_form}) are discussed 
in~\cite{ilin91,ilin94, ilin_filatov}. 
A regular solution of the Navier-Stokes equations (\ref{NSE_sphere}) on
$[0,T]$ is a tangential divergence-free velocity field $\vu$
that satisfies  the equation obtained by integrating in time the weak form  
 (\ref{weak_form}), from $t_0$ to $t$, for almost every 
$t_0, t \in [0,T]$.
In order to recall the existence, uniqueness, and 
Gevrey regularity of the regular solution,
we need a few more additional  details from~\cite{cao_rammaha_titi99}. 
These are also needed as tools for analyzing our 
pseudospectral quadrature method. 
\subsection{Gevrey regularity of  regular solution}
The Gevrey class of functions of order $s>0$ and index $\sigma>0$,
associated with the Stokes operator defined in (\ref{A}),  is denoted by
  $G^{s/2}_\sigma$ and is defined as 
\begin{equation}\label{Gev_space}
  G^{s/2}_\sigma := \calD(\A^{s/2} e^{\sigma A^{1/2}}) \subset\calD(\A^{s/2}).
\end{equation}
Using (\ref{Gev0_space}), the Gevrey space
\begin{equation}\label{Gev_space_char}
G^{s/2}_\sigma = 
\left\{\vv \in \calD(\A^{s/2}) \ : \vv = \sum_{L=1}^\infty \sum_{m=-L}^{L} \widehat{\vv}_{L,m} \Z_{L,m},
              \;\sum_{L=1}^\infty \sum_{m=-L}^L \lambda^s_L
              e^{2\sigma\lambda^{1/2}_{L}}  |\widehat{\vv}_{L,m}|^2<\infty 
               \right\}
\end{equation}
is a Hilbert space with respect to the inner product
\begin{equation}\label{Gev_ip}
 {\inprod{\vv}{\vw}}_{G^{s/2}_\sigma} = 
\sum_{L=1}^\infty \sum_{m=-L}^L  \lambda^s_L e^{2\sigma\lambda^{1/2}_{L}} 
     \widehat{\vv}_{L,m} \overline{\widehat{\vw}_{L,m}},
    \qquad  \vv,\vw \in G^{s/2}_\sigma.
\end{equation}
First we recall the following result from \cite{cao_rammaha_titi99, Titi2009}. 
\begin{theorem}\label{reg_thm}
If $\vu_0 \in \calD(\A^{s+1/2})$ and
$\vf \in L^\infty((0,\infty);\calD(\A^s e^{\sigma_1 \A^{1/2}}))$,
for some  $s, \sigma_1 > 0$,  then  for all $t>0$ there exists
a $T^* > 0$, depending only on $\nu,\vf$, and  
$\|\A^{s+1/2} \vu_0\|_{L^2(TS)}$, such that 
the NSE (\ref{NSE_sphere}) on $S$ have a unique
regular solution $\vu(\cdot, t)$ 
and $\vu(\cdot,t) \in G^{s+1/2}_{\sigma(t)}$, 
where $\sigma(t) = \min\{t,T^*,\sigma_1\}$.
\end{theorem}
In addition, from the assumption and proof of Theorem~\ref{reg_thm} 
\cite[page 355]{cao_rammaha_titi99}, 
\be\label{bounded_soln}
 \| \A^{s+1/2} e^{\sigma(t)\A^{1/2}} \vu(t)\|^2
  \le M_0, \qquad t>0,
\ee
where $M_0$ depends on $\|\A^{s+1/2}\vu(0)\|$, 
$\sup_{t\ge 0}\|\A^s \vf(\cdot,t)\|$ and $\nu$ but not on $t$.
The bound in (\ref{bounded_soln}) is useful for establishing
the quality of approximation of the Stokes projection  of $\vu$ in the next 
section (see Theorem~\ref{cont_est}). It is also convenient
to have a similar bound for the time derivative of the Stokes
projections with $t$ in (\ref{bounded_soln}) 
replaced with certain complex times $\zeta \in \C$,  to
prove  the power of approximation of the time derivative 
of the Stokes projection  (see Theorem~\ref{der_cont_est}).
To this end, we consider the NSE extended to 
complex times $\zeta$,
\be\lb{NSE_complex}
\frac{d\vu}{d\zeta} + \nu \A \vu + \B(\vu,\vu) + \CC\vu = \vf,
\quad \Div\vu=0, \quad \vu(0)=\vu_0, \qquad \zeta \in \C,  \qquad \text{on}~~S ,
\ee
with standard complexification (see~\cite{Foias_book}) of all the spaces and 
operators introduced earlier. 
In the next theorem, we extend arguments 
used in~\cite{Foias_book}  
for the solution of the NSE on 
the plane to the case of the sphere. The arguments
differ in an essential way only for the nonlinear term.

\begin{theorem}\label{complex_est}
Let $\vu_0 \in \calD(\A^{s+1/2})$ and
$\vf \in C([0,T];\calD(\A^{s} e^{\sigma_1 \A^{1/2}}))$, with
$s~\geq~1/4$. Let the domain $\T$ be defined by
\[
\T :=  \{\zeta=r e^{i\theta}: 0 \leq r \leq T;\ |\theta| \le \pi/4\}.
\]
We assume further that $\vf(\cdot,\zeta)$ is analytic for 
$\zeta\in\T$ and that
\be\label{supf_in_T}
K := 
\sup \{ \|\Ape \vf(\cdot,\zeta)\|^2 : \zeta=re^{i\theta} \in \T\} < \infty.  
\ee
Then there exists $T^{**} > 0$  such that
\be\label{bounded_soln-complex-local}
 \| \A^{s+1/2} e^{\psi(r\cos\theta)\A^{1/2}} \vu(\zeta) \|^2
   \le M_1,  \qquad \zeta \in \T \mbox{ and } |\zeta| \le T^{**},  
\ee
where $M_1$ depends on $\|\A^{s+1/2} \vu(0)\|$ and 
hence $\vu(\cdot,\zeta) \in 
G^{s+1/2}_{\psi(r\cos\theta)}$, where \\ 
$\psi(x)~:=~\min\{x,T^{**},\sigma_1\}$. 
\end{theorem}
\begin{proof}
Let $\zeta = r e^{i\theta}$ with $r > 0$ and $|\theta| \leq  \pi/4$, 
and let 
\[
\vu_{\A}(\zeta) :=  \Aphalfe \vu(\zeta).
\]
The definition of $\psi$ gives 
$\frac{d}{dx} \psi  :=  \psi'(x) \le 1$, and $\psi \le
\sigma_1$, thus for fixed $\theta$ we find
\bea
\frac{d}{dr} \vu_{\A}(\zeta)  & = & 
\psi'(r\cos\theta) \cos\theta  \Apone \vu(\zeta) 
\nonumber \\ 
& & 
+ e^{i\theta} \Aphalfe \frac{d\vu}{d\zeta}(\zeta). \label{der2}
\eea
Using \eqref{der2} in 
\beas\lb{der1}
\frac{1}{2}\frac{d}{dr} \left\|\vu_{\A}(\zeta) \right\|^2 =
\Re\left(\frac{d}{dr} \vu_{\A}(\zeta), \vu_{\A}(\zeta) \right),
\eeas
where $\Re(\cdot)$ denotes the
real-part function,  we get
\bea
\frac{1}{2} \frac{d}{dr} \| \vu_{\A}(\zeta)\|^2  
&=&
\psi'(r\cos\theta)\cos\theta\; 
 \Re\;(\A^{1/2} \vu_{\A}(\zeta), \vu_{\A}(\zeta)) 
\nonumber  \\ &&
+\Re\; e^{i\theta} (\Aphalfe \frac{d\vu}{d\zeta}(\zeta) , \vu_{\A}(\zeta)). 
\lb{der3} \eea
Using \eqref{NSE_complex}  
for the last term in \eqref{der3} 
together with 
the fact that 
\[
(\A^s e^{\psi(r\cos\theta)\A^{1/2}}\CC\vu,\Apone\vu) = 0
\]
(see \cite[Lemma 1]{cao_rammaha_titi99}), we find 
\bea
& & \frac{1}{2} \frac{d}{dr} \| \vu_{\A} (\zeta) \|^2
+ \nu \cos\theta \|\A^{1/2} \vu_{\A}(\zeta)  \|^2 
\nonumber \\  & = &  
\psi'(r\cos\theta) \cos\theta(\A^{1/2} \vu_{\A} ,\vu_{\A} )
  - \Re\; e^{i\theta} ( \Aphalfea \B(\vu,\vu),\vu_{\A}) 
\nonumber\\ & &  
+ \Re\; e^{i\theta} (\Apea\vf, \A^{1/2} \vu_{\A}),  \label{main_ineq}
\eea
where in \eqref{main_ineq} and below we write 
$\vu_{\A} = \vu_{\A}(\zeta),  \vu = \vu(\zeta), 
\vf = \vf(\zeta),  \psi =  \psi(r \cos\theta)$.
From \cite[Lemma 2]{cao_rammaha_titi99}, we have
(with $p=\max\{2-s,7/4\} = 7/4$, since $s \geq 1/4$),
\be\label{Buu_est}
|( \Aphalfea \B(\vu,\vu),\vu_{\A})| 
  \le C \| \vu_{\A}\|^{3-p} 
        \| \A^{1/2} \vu_{\A}\|^p.
\ee
Applying  $\psi' \leq 1$, the Cauchy-Schwarz inequality, \eqref{Buu_est} 
and Young's inequality ($ab \le a^q/q + b^{q'}/q'$ with $1/q+1/q'=1$)  
with $q=2$ and $q=2/(2-p)$
in \eqref{main_ineq}, we get 
\beas
& & \frac{1}{2} \frac{d}{dr} \| \vu_{\A} \|^2
+ \nu \cos\theta \|\A^{1/2} \vu_{\A}  \|^2 \\
&\le& \cos\theta \| \A^{1/2} \vu_{\A} \| \| \vu_{\A}  \| 
 + C \|\vu_{\A}   \|^{3-p} \|\A^{1/2} \vu_{\A}  \|^p 
 + \|\Apea \vf\| \| \A^{1/2} \vu_{\A} \| \\
&\le& \frac{\nu\cos\theta}{4} \| \A^{1/2} \vu_{\A} \|^2 +
       \frac{\cos\theta}{\nu} \|\vu_{\A} \|^2 \\
&&+ C^{\frac{2}{2-p}} \frac{(2-p)}{p}
    \left(\frac{p}{\nu\cos\theta}\right)^{\frac{p}{2-p}} 
 \| \vu_{\A}  \|^{\frac{2(3-p)}{2-p}} + 
  \frac{\nu\cos\theta}{2} \|\A^{1/2} \vu_{\A}  \|^2\\
&&+\frac{1}{\nu\cos\theta}\|\Apea \vf\|^2+
 \frac{\nu\cos\theta}{4} \| \A^{1/2} \vu_{\A} \|^2.
\eeas
Therefore,
\bea 
 \frac{d}{dr} \|\vu_{\A}    \|^2
&\le& \frac{2 \cos\theta}{\nu} \|\vu_{\A} \|^2 
 + C \left(\frac{1}{\nu\cos\theta}\right)^{\frac{p}{2-p}} 
 \| \vu_{\A} \|^{\frac{2(3-p)}{2-p}}
 + \frac{2}{\nu\cos\theta} \| \Apea \vf \|^2 \nonumber \\
&\le& \frac{2 \cos\theta}{\nu} (1+\|\vu_{\A} \|^2) +
      C \left(\frac{1}{\nu\cos\theta}\right)^{\frac{p}{2-p}}
       (1+\| \vu_{\A} \|^2)^{\frac{3-p}{2-p}} + \nonumber\\ 
& &     \frac{2}{\nu\cos\theta}\|\Apea \vf\|^2 .
       \lb{eq:ddr_bd}
\eea
Using $|\theta| \le \pi/4$,  
and $\frac{3-p}{2-p}= 5$,  
in \eqref{eq:ddr_bd}, we get 
\be \lb{eq:ddr_bd1}
 \frac{d}{dr} \|\vu_{\A} \|^2
\le C (1+\|\vu_{\A}\|^{2})^5 + \frac{2\sqrt{2}}{\nu} K.
\ee
With $|\theta| \le \pi/4$ fixed, let
\[
  y(r) = 1 + \|\vu_{\A}\|^2.
\]
Then on using \eqref{eq:ddr_bd1} and $y\ge 1$ we obtain 
\[
  \frac{d}{dr} y \le C y^{5},
\]
and hence on integrating the inequality we find that
\[
  y(r) \le 2 y(0) 
\]
provided that 
\[
  0 \le r \le \frac{15}{64C} \frac{1}{(y(0))^4} 
       = \frac{15}{64C} \frac{1}{ (1+\|\A^{s+1/2} \vu(0)\|^2)^4 }.
\]
By setting
\[
  T^{**}:= \frac{15}{64C} \frac{1}{ (1+\|\A^{s+1/2}\vu(0)\|^2)^4 },\quad 
  M_1 = 1+2\|\A^{s+1/2}\vu(0)\|^2,
\]
we deduce that 
\eqref{bounded_soln-complex-local} holds for $0 \leq  r \le T^{**}$. 
\end{proof}
We can extend the bound \eqref{bounded_soln-complex-local} to a
larger domain which contains the interval $[0,T]$ using the
following property of the NSE solution on the sphere \cite{Titi2009}:
\be\label{ass:solutionboundedT}
  \| \A^{s+1/2} \vu(\cdot,t) \| \le M_2 \mbox{ for all } t \in [0,T],
\ee
where the constant $M_2$ depends only on $\|\A^{s+1/2}\vu_0\|$,
$\sup_{0\le t \le T} \|\A^{s} \vf(\cdot,t)\|$ and $\nu$ but not on $T$.
\begin{theorem}\label{complex_estT}
Suppose $\vu_0$ and $\vf$ satisfy all the conditions in 
the domain $\T$ as in Theorem \ref{complex_est}.  
Then 
\be\label{bounded_soln-complex}
 \| \A^{s+1/2} e^{\psi(r\cos\theta)\A^{1/2}} \vu(\zeta) \|^2
    \le M_3,  \qquad \zeta \in \T \mbox{ and } 
    |\mbox{Im }\zeta| \le T^{**}/\sqrt{2},
\ee
where $M_3:=1+2M_2^2$, $T^{**}$ depends on
$M_2$, and hence $\vu(\cdot,\zeta) \in G^{s+1/2}_{\psi(r\cos\theta)}$
for all $\zeta \in \T$ with $|\mbox{Im }\zeta| \le T^{**}/\sqrt{2}$.
\end{theorem}
\begin{proof}
We proceed as in the proof of Theorem \ref{complex_est} to obtain
the ordinary differential equation
\[
  \frac{d}{dr} y \le C y^5,
\]
where $y(r) = 1+\|\Aphalfe\vu(r e^{i\theta})\|^2$. On integrating
the inequality we find that
\[
  y(r) \le 2 y(0)
\]
provided that
\[
  0 \le r \le \frac{15}{64C} \frac{1}{ (1+\|\A^{s+1/2} \vu(0)\|^2)^4 }.
\]
We define
\[
T_1(\rho):= \frac{15}{64C} \frac{1}{ (1+\rho^2)^4 },\quad \rho \ge 0.
\]
If $T_1(\|\A^{s+1/2}\vu(0)\|) \ge T$ we have finished the proof. Otherwise, 
we let
\[
  T^{**} = T_1(M_2),
\]
where $M_2$ is given in \eqref{ass:solutionboundedT}.
For $\zeta = r e^{i\theta} \in \T$, $0 \leq r \le T^{**}$, 
\begin{equation}\label{eq:uT**-bdd}
 \|\Aphalfe \vu(r e^{i\theta}) \|^2 \le 1 + 2\|\A^{s+1/2} \vu(0)\|^2.
\end{equation} 
Consequently, \eqref{bounded_soln-complex} holds for $0 \leq  r \le T^{**}$
with $M_3 = 1+2\|\A^{s+1/2} \vu(0)\|^2$. 

Next we consider the case $\zeta =  T^{**}+ r e^{i\theta}$, with 
$r \in [0,T^{**}]$. We define, 
for $|\theta| \leq \pi/4$,
\[
\vv(re^{i\theta}) :=  \vu(T^{**}+re^{i\theta}), \qquad r \in [0,T^{**}].
\]
Using 
$\|\A^{s+1/2} \vv(0)\| \le M_2$ we can apply 
the previous arguments 
to obtain \eqref{bounded_soln-complex} (with $\vu$ replaced with $\vv$)
for $0\le r\le T^{**}$. 
We complete the proof and obtain the bound \eqref{bounded_soln-complex} 
by repeating the last argument $n$ times, where $n=\lceil T/T^{**} \rceil$. 
\end{proof}
\section{Finite dimensional spaces and Stokes projections}~\label{sec:proj} 
Throughout the remainder of the paper, with $s$ and $\sigma_1$ as in 
Theorem~\ref{reg_thm} and Theorem~\ref{complex_est}, we assume that 
\begin{equation}\label{main_assumption}
\vu_0 \in \calD(\A^{s+1/2}), \qquad  
\vf \in C([0,T];\calD(\A^{s+1/2} e^{\sigma_1 \A^{1/2}})), 
\qquad s \geq 1/4. 
\end{equation}

Natural finite dimensional spaces (depending on a parameter $N>0$) in
which to seek approximations to $\vu(t)$ are
\begin{equation} \label{fin_space_vec}
   V_N := \mbox{ span } \{ \Z_{L,m}: L=1,\ldots,N;~~ m=-L,\ldots,L\}.
\end{equation}
The dimension of $V_N$ is $N^2+2N$. Let  $\proj_N:H \rightarrow V_N$ 
be the orthogonal projection with respect to the $L^2(TS)$ inner 
product defined by 
\begin{equation}\label{vec_orth_proj}
   \proj_N(\vv) = \sum_{L=1}^N \sum_{m=-L}^L \widehat{\vv}_{L,m} \Z_{L,m}.
\end{equation}
\begin{lemma}
Let $\alpha>0$ be given. If $\vv \in {\mathcal D}(\A^{\alpha/2})$ then
\begin{equation}\label{app_power_veca}
 \|\vv - \proj_N \vv\| \le  N^{-\al} \|\vv\|_{H^\alpha(TS)}.
\end{equation}
\end{lemma}
\begin{proof}
Using (\ref{stok_eig}), (\ref{Fourier_coeff}), and 
(\ref{Hs-norm})   we get
\begin{align*}
 \|\vv - \proj_N \vv\|^2=
   \sum_{L=N+1}^\infty\sum_{m=-L}^L |\widehat{\vv}_{L,m}|^2
& \le 
 N^{-2\alpha}
 \sum_{L=N+1}^\infty\sum_{m=-L}^L\lambda^{\alpha}_L |\widehat{\vv}_{L,m}|^2 \\
& \le N^{-2\alpha}\| \vv \|^2_{H^\alpha(TS)}.
\end{align*}
\end{proof}
In particular, using \eqref{main_assumption}, we get
\begin{equation}\label{app_power_vec}
 \|\vf - \proj_N \vf\| \le  N^{-(2s+1)} \|\vf\|_{H^{2s+1}(TS)}, \qquad t \in [0,T].
\end{equation}

For computer implementation,
the Fourier coefficients in (\ref{vec_orth_proj}) and all Galerkin
type integrals in computational schemes for the NSE need to be 
approximated by cubature rules on the sphere, leading to 
a pseudospectral method. To this
end,  for a continuous scalar field $\psi$ on $S$, we consider
a  Gauss-rectangle quadrature sum $Q_M(\psi)$ 
with quadrature points $\{\hxi_{p,q} = \p(\theta_p,\phi_q)\}$ 
and  positive weights ${w_p}$ of the form 
\be\label{quadrule}
   Q_M (\psi) := \frac{2\pi}{M}\sum_{q=1}^M \sum_{p=1}^{M/2} w_p \psi(\hxi_{p,q}) =  
\frac{2\pi}{M}\sum_{q=1}^M \sum_{p=1}^{M/2} w_p \psi(\theta_p, \phi_q),
\ee
where $M \ge 2$ is an even integer, $w_p$ and $\cos\theta_p$
for $\quad p=1,\ldots,M/2$ are 
the Gauss-Legendre weights and nodes on $[-1,1]$ and 
$\phi_q=2q\pi/M$, $q=1,\ldots,M$. 
The rule (\ref{quadrule})  is
exact when $\psi$ is a polynomial of degree $M-1$ on $S$, that is,
\[
Q_M \psi = \int_{S} \psi\; dS, \qquad \psi \in \calP_{M-1}.
\]
Hence corresponding to (\ref{scal_ip}) and (\ref{vec_ip}), we define discrete
inner products for scalar and vector fields on the unit sphere  as
\be \label{disc_ip}
(v_1,\;v_2)_M = Q_M(v_1v_2), \qquad \qquad (\vv_1,\;\vv_2)_M = Q_M(\vv_1 \cdot \vv_2).
\ee
The choice of $M$ is very important; we choose $M$ such that all Galerkin integrals
with polynomial terms in our scheme are evaluated exactly. In particular, with
the unknown tangential divergence-free velocity field sought in the polynomial 
space $V_N$, and knowing that the NSE nonlinearity is 
quadratic, we choose $M$ such that  
\be \label{M_choice}
  (\B(\vv,\vw), \ovz) = (\B(\vv,\vw),\ovz)_{M},
 \qquad  \vv, \vw, \ovz \in V_N. 
\ee
This holds, for example, if  $M = 3N+2$. We define a computable
counterpart of (\ref{vec_orth_proj}), using 
$\dproj_N: H \cap C(TS) \rightarrow V_N$,
a discrete orthogonal projection with respect to the $M^2/2$ point 
discrete inner  product, as 
\begin{equation}\label{disc_vec_orth_proj}
   \dproj_N(\vv) = \sum_{L=1}^N \sum_{m=-L}^L \widehat{\vv}_{L,m,M} \Z_{L,m},
   \qquad \widehat{\vv}_{L,m,M} = Q_M( \vv \cdot \overline{\Z_{L,m}}) = (\vv,\;\Z_{L,m})_M .
\end{equation}
With $M$ chosen to satisfy (\ref{M_choice}), it is easy to see that
\begin{equation} \label{disc_proj_prop1}
 \dproj_N(\vv) = \proj_N(\vv) = \vv,  \qquad \vv \in V_N, 
\end{equation}
and for $\vv \in H\cap C(TS)$ and
$\vw \in H\cap C^k(TS)$,
\begin{equation} \label{disc_proj_prop2a}
\| \dproj_N(\vv)\| \leq C \|\vv\|_\infty, 
\qquad \qquad
\| \vw - \dproj_N(\vw)\| \leq C N^{-k} \|\vw\|_{C^k(TS)},
\end{equation}
where the last two inequalities follow from simple arguments
used in Theorem 13 and Lemma 14 of \cite{vec_hyp_pap}. In particular,
since  $\calD(\A^{s+1/2}) \subset  H\cap C^{2s}(TS)$, for an integer $2s$,
using \eqref{main_assumption}  
\begin{equation} \label{disc_proj_prop2}
\| \vf - \dproj_N(\vf)\| \leq C N^{-2s} \|\vf\|_{C^{2s}(TS)}, \qquad t \in [0,T].
\end{equation}

Next we consider the Stokes projection 
in $V_N$ of the exact unique regular solution 
$\vu(t) := \vu(.,t)$ of (\ref{NSE_sphere}). 
For each fixed $t$, the Stokes projection $\widetilde{\vu}_N \in V_N$ of $\vu$
is defined by
\be \label{elliptic1}
  (\A\widetilde{\vu}_N,\vv) = (\A\vu,\vv), \qquad \vv \in V_N.
\ee
Since $\proj_N \A = \A \proj_N$, it follows that 
$\widetilde{\vu}_N = \proj_N (\vu)$.
Following standard techniques in finite element analysis, 
the Stokes projection of $\vu$ plays an important role as a comparison  
function in the main analysis in the next section.
\begin{theorem}\label{cont_est}
Let $\vu_0$ and $\vf$ satisfy \eqref{main_assumption}.  
Then  
\begin{equation} \label{comp_fn_bound}
  \|\vu-\tilde{\vu}_N\| 
  \le   C  \lambda^{-s-1/2}_{N+1} e^{-\sigma(t) \lambda^{1/2}_{N+1}}
  \le   C  N^{-2s-1} e^{-\sigma(t) N}, \qquad  t \in [0, T],
\end{equation}
where $\sigma(t)$  is as in Theorem \ref{reg_thm}.
\end{theorem}
\begin{proof}
Using the fact that $\tilde{\vu}_N = \proj_N \vu$, we have
\beas
\|\vu - \tilde{\vu}_N \|^2 
&=& \sum_{L>N} \sum_{|m| \le L} |\widehat{\vu}_{L,m}|^2 \\
&\le& \lambda^{-2s-1}_{N+1} e^{-2\sigma(t)\lambda^{1/2}_{N+1}}
     \sum_{L>N} \sum_{|m|\le L} \lambda^{2s+1}_L 
       e^{2\sigma(t) \lambda^{1/2}_L } |\widehat{\vu}_{L,m}|^2 \\
&\le& \lambda^{-2s-1}_{N+1} e^{-2\sigma(t)\lambda^{1/2}_{N+1}}
      \| \A^{s+1/2} e^{\sigma(t)\A^{1/2}} \vu(t) \|^2 \\ 
&\le& C \lambda^{-2s-1}_{N+1} e^{-2\sigma(t)\lambda^{1/2}_{N+1}}
\eeas
where in the last step we used \eqref{bounded_soln}. The 
last inequality in (\ref{comp_fn_bound}) follows from 
the fact that $N^2 \leq \lambda_{N+1} = (N+1)(N+2)$. 
\end{proof}
\begin{theorem}\label{der_cont_est}
Let $\vu_0$, $\vf$ satisfy \eqref{main_assumption}. 
We assume further that $\vf$ is analytic 
in $\T$ and \eqref{supf_in_T} holds.
Then for $t \in (0, T)$, 
\begin{equation} \label{der_comp_fn_bound}
  \left\|\frac{d}{dt}\left(\vu-\tilde{\vu}_N\right)\right \| \le 
C  \lambda^{-s-1/2}_{N+1} e^{-\psi_1(t) \lambda^{1/2}_{N+1}} \le
C  N^{-2s-1} e^{-\psi_1(t) N},  
\end{equation}
where $\psi_1(t) = \min\{(1-1/\sqrt{2})t,T^{**},\sigma_1\}$, and $T^{**}$   
is as in Theorem~\ref{complex_estT}.
\end{theorem}
\begin{proof}
Let $t \in (0, T)$ be fixed. Let 
$\vp_N(\zeta) = (\vu - \tilde{\vu}_N)(\zeta)$ be the standard 
complexification
of $\vu - \tilde{\vu}_N$ at $\zeta = r e^{i\theta}$.   
Using Theorem~~\ref{complex_estT} 
and the Cauchy integral formula,
\[
 \frac{d\vp_N(t)}{dt}=
 \frac{1}{2\pi i } \int_{\Gamma} \frac{\vp_N(\zeta)}{(t-\zeta)^2} d\zeta,
 \]
where for $t>0$, $\Gamma$ is a circle in the $\zeta$ plane
with center $(t,0)$ and 
radius\\
$\min\{t/\sqrt{2},T^{**}/\sqrt{2},T-t\}$, 
a condition that ensures that $\zeta=r e^{i\theta} \in \Gamma$ 
lies in the region $\T$ with 
$|\mbox{Im }\zeta| \le T^{**}/\sqrt{2}$. 
Using the fact that $\tilde{\vu}_N = \proj_N \vu$, 
for $\zeta = r e^{i\theta} \in \Gamma$ we have, 
from Theorem~\ref{complex_estT}, 
\bea
\|\vp_N(\zeta)\|&=& \|\vu(\zeta) - \tilde{\vu}_N(\zeta) \|^2 \nonumber \\ 
& = &  \sum_{L>N} \sum_{|m| \le L} |\widehat{\vu}_{L,m}|^2 \nonumber\\
&\le& \lambda^{-2s-1}_{N+1} e^{-2\psi(r\cos \theta)\lambda^{1/2}_{N+1}}
   \sum_{L>N} \sum_{|m|\le L} \lambda^{2s+1}_L 
   e^{2\psi(r\cos \theta) \lambda^{1/2}_L } |\widehat{\vu}_{L,m}|^2 \nonumber \\
&\le&  \lambda^{-2s-1}_{N+1} e^{-2\psi(r\cos \theta)\lambda^{1/2}_{N+1}}
      \| \A^{s+1/2} e^{\psi(r \cos \theta)\A^{1/2}} \vu(\zeta) \|^2 \nonumber \\ 
&\le& C \lambda^{-2s-1}_{N+1} e^{-2\psi(r\cos \theta)\lambda^{1/2}_{N+1}}
         \nonumber\\ 
&\le&  C N^{-2(2s+1)}  e^{-2\psi(r\cos\theta) N}.
\label{eq:pn-comp-est}
\eea
For $\zeta = r e^{i\theta} \in \Gamma$ it is easily seen that
$r \cos\theta \ge (1-1/\sqrt{2})t$, and hence that
\be\lb{psi1_def}
 \psi(r\cos\theta) \ge \min \{(1-1/\sqrt{2})t,T^{**},\sigma_1\} =: \psi_1(t).
\ee
On using (\ref{eq:pn-comp-est}) in  (\ref{comp_fn_bound}) we get 
 \[
  \left\|\frac{d}{dt}\left(\vu-\tilde{\vu}_N\right)(t)\right\| \le 
 \frac{1}{2\pi}\int_{\Gamma} \frac{\|\vp_N(\zeta)\|}{|t-\zeta|^2} d\zeta \le 
C  N^{-2s-1} e^{-\psi_1(t) N}.
\]
\end{proof}
\section{A  pseudospectral quadrature  method}
We are now ready to describe, analyze, and implement a spectrally
accurate scheme to compute approximate solutions of the NSE (\ref{NSE_sphere})
in $V_N$, through its weak formulation (\ref{weak_form}). The
task is then to compute 
$\vu_N(\cdot,t) \in V_N$ for $t \in [0,T]$ with 
$\vu_N(0,\xh)= \dproj_N\vu_0(\xh)$, $\xh \in S$, satisfying  
the (spatially)  discrete system of ordinary differential 
equations 
\begin{equation} \label{app_weak_form}
  \frac{d}{dt}\left(\vu_N,\vv\right)_M + b(\vu_N,\vu_N,\vv)_M 
+ \nu(\Curln \vu_N, \Curln \vv)_M 
+ (\CC\vu_N, \vv)_M  = (\vf,\vv)_M,
\end{equation}
for all $\vv \in V_N$ and prove that the scheme is spectrally accurate 
with respect to 
the parameter $N$ (that is, converges with  rate determined by 
the smoothness of the given data), and demonstrate the theory 
with numerical experiments.

Using (\ref{fin_space_vec}), 
the exactness properties of the discrete inner product,  (\ref{M_choice}),
(\ref{ip_identities1}), (\ref{A}),  (\ref{C}), and (\ref{B}), 
the system (\ref{app_weak_form}) can be written as
\bea 
\left( 
\frac{d}{dt} \vu_N + \nu\A \vu_N + \B(\vu_N,\vu_N)+\CC\vu_N,\Z_{L,m}
\right)
&=& (\vf,\Z_{L,m})_M, \nonumber \\ 
& & \hspace{-1.7in} L=1,\cdots,N;~~ m=-L,\cdots,L. \label{app1_weak_form}
\eea

\subsection{Stability and convergence analysis}\label{sec:anal}
First we establish the stability of the approximate 
solution $\vu_N$ of~\eqref{app_soln_rep}. That is, similar 
to~\cite[Proposition~9.1]{constantin_foias},  
we prove that $\max_{t \in  [0, T]} \|\vu_N\|_V$ is uniformly bounded 
with the bound depending only on the initial data, forcing
function, and the viscosity term in~\eqref{app_soln_rep}.

\begin{theorem}\label{stability_thm}
Let  $\vu_0$ and $\vf$ satisfy (\ref{main_assumption}).
Let $N\ge 1$ be an integer. 
Let $\vu_N$ be the solution of the pseudospectral quadrature
Galerkin system~\eqref{app_weak_form}. Then there
exists a constant $C$ depending on $\nu, \|\vu_0\|_V$
and $\|\vf\|_{\infty} := \max_{t \in [0, T] } \|\vf(t)\|_{C(TS)}$ so that
\[
  \max_{t \in [0, T]} \|\vu_N\|_V \le C.
\]
\end{theorem}
\begin{proof}
The proof follows by repeating  the arguments described 
in the first four pages 
of~\cite[Section~9]{constantin_foias}
(proving~\cite[Proposition~9.1]{constantin_foias}),
provided that we
establish~\cite[Inequalities~(9.3)~and~(9.13)]{constantin_foias}  
for our system~\eqref{app_weak_form} on the spherical surface
with additional Coriolis term and quadrature approximations. 

Using~\eqref{skew}, \eqref{B} and the exactness
of the quadrature rule (given by \eqref{disc_ip}-\eqref{M_choice}), 
we have $\left(\B(\vu_N,\vu_N),\vu_N\right)_M = 0$.
Using \eqref{Coriolis_ZLm}, the symmetry of the coefficients of $\vu_N$ 
in \eqref{app_soln_rep} and the exactness of the quadrature, we get  
\begin{equation}\lb{CuNuN}
(\CC\vu_N,\vu_N)_M=
(\CC\vu_N,\vu_N) = 
(-2\Omega i) 
\sum_{L=1}^N \lambda^{-1}_L \sum_{|m|\le L} m|\alpha_{L,m}|^2 
= 0. 
\end{equation}
By taking $\vv$ to be $\vu_N$ in \eqref{app_weak_form} and 
using   $\|\vu_N\|^2 = (\vu_N,\vu_N)_M$,
$\|\vu_N\|_V^2 = (\A \vu_N,\vu_N)_M$,  \eqref{CuNuN}, 
Young's inequality and the fact that all the eigenvalues  
$\lambda_J$ of $\A$ 
(corresponding to eigenvectors in $\vu_N$) satisfy 
$\lambda_J \geq \lambda_1 = 2, J = 1, \cdots, N$,
we obtain  
\[
\frac{1}{2} \frac{d}{dt} \|\vu_N\|^2
+ \nu \|\vu_N\|_V^2 
= (\vf,\vu_N)_M 
\le \|\vf\|_{\infty}  \|\vu_N\|
\le  \frac{ \|\vf\|^2_{\infty} } {4\nu} 
   + \nu \|\vu_N\|^2 \leq \frac{ \|\vf\|^2_{\infty} } {4\nu} + \frac{\nu}{2} \|\vu_N\|_V^2,
\]
Hence, for our discrete system~\eqref{app_weak_form}, 
we obtain~\cite[Inequality~(9.3)]{constantin_foias}: 
\begin{equation}\label{ode_ineq}
 \frac{d}{dt} \|\vu_N \|^2 + \nu \|\vu_N\|_V^2 
  \le \frac{\|\vf\|^2_{\infty}}{\nu \lambda_1}.
\end{equation}
Again using \eqref{Coriolis_ZLm}, the exactness of the
quadrature rule, eigenfunction properties of $\A$, and 
the symmetry of the coefficients of $\vu_N$ in 
\eqref{app_soln_rep}, we get 
\begin{equation}\lb{CuNAuN}
(\CC\vu_N,\A\vu_N) = 
(-2\Omega i) 
\sum_{L=1}^N  \sum_{|m|\le L} m|\alpha_{L,m}|^2 
=0, 
\end{equation}
By taking $\vv$ to be $\A\vu_N$ in equation \eqref{app_weak_form},
and  using   $\|\A\vu_N\|^2 = (\A\vu_N,\A\vu_N)_M$, \eqref{M_choice}, and  \eqref{CuNuN}, 
we obtain 
\be\label{inprodAuN}
\frac{1}{2} \frac{d}{dt} \|\vu_N\|^2_V +
\nu \|\A \vu_N\|^2 + b(\vu_N,\vu_N,\A\vu_N) = (\vf,\A\vu_N)_M.
\ee
Using~\cite[Lemma~3.1]{ilin91},  $b(\vu_N,\vu_N,\A\vu_N) = 0$. 
The term $(\vf,\A\vu_N)_M$ can be estimated  by using the exactness
of the quadrature and Young's inequality:
\[
(\vf,\A\vu_N)_M \le \|\vf\|_{\infty}  \|\A\vu_N\|
\le  \frac{ \|\vf\|^2_{\infty} } {2 \nu} 
   + \frac{\nu}{2} \|\A\vu_N\|^2
\] 
Hence we obtain a stronger version of ~\cite[Inequality~(9.13)]{constantin_foias} for
our quadrature discrete scheme~\eqref{app_weak_form}: 
\[
 \frac{d}{dt} \|\vu_N\|^2_V + \nu \|\A\vu_N\|^2
 \le  \frac{ \|\vf\|^2_{\infty} } {\nu \lambda_1}.
\]
Thus, the result follows from arguments in~\cite[Page~74-77]{constantin_foias}.
\end{proof}

Next, using 
Theorem~\ref{cont_est},~\ref{der_cont_est}, and ~\ref{stability_thm}, we prove the spectral convergence of
the solution $\vu_N$ of  (\ref{app1_weak_form}) to the solution
of  $\vu$ of (\ref{weak_form}).
\begin{theorem}\label{conv_anal}
Let  $\vu_0$,  $\vf$ satisfy 
\eqref{supf_in_T}, \eqref{main_assumption}.
Then there exists a $T^{\#}> 0$, 
depending only on $\nu,\vf$, $\vu_0$ and the uniform
bound in \eqref{ass:solutionboundedT}
(and hence there exists $0 < \mu(t) < \min\{t,T^{\#},\sigma_1\}$) 
such that for all $ t \in (0,T)$, 
\begin{equation} \label{final_bound}
  \|\vu-\vu_N\| \le   C \left[ N^{-2s-1} e^{-\mu(t) N}
+  \left\|\proj_N \vf - \dproj_N \vf\right\| \right].
\end{equation}
In particular, with  $2s$ being an integer
\begin{equation}\label{int_bound}
  \|\vu-\vu_N\| \le   C N^{-2s}.
\end{equation}
\end{theorem}
\begin{proof}
Let $\vw_N = \tilde{\vu}_N - \vu_N$, where the comparison
function  $\tilde{\vu}_N$ is the
solution of (\ref{elliptic1}). Since  
$\vu - \vu_N = \vp_N  + \vw_N$, where
$\vp_N = \vu - \tilde{\vu}_N$, in view of
Theorem~\ref{cont_est} and \ref{der_cont_est},
existence of $T^{\#}$ and
$\mu(t)$ follows and it is sufficient to show that 
$ \|\vw_N\| \le   C  \left[ N^{-2s-1} e^{-\mu(t) N} +
 \left\|\proj_N \vf - \dproj_N \vf\right\| \right]
 $.
For  any  $\vv \in V_N$, using (\ref{op_form}), (\ref{elliptic1}), 
and (\ref{app1_weak_form}),
\beas
&& \hspace{-0.2in} ((\vw_N)_t, \vv) + \nu (\A\vw_N, \vv) + (\CC\vw_N,\vv) \\ 
&& \hspace{-0.2in} = 
      ((\tvuN)_t,\vv)+\nu (\A\tvuN,\vv)+(\CC\tvu_N,\vv) 
    -((\vu_N)_t,\vv)-\nu(\A\vu_N,\vv)-(\CC\vu_N,\vv) \\
&& \hspace{-0.2in} =  
((\tvuN)_t,\vv)+\nu(\A\tvuN,\vv)+(\CC\tvu_N,\vv) 
     -(\vf,\vv)_M + (\B(\vu_N,\vu_N),\vv) \\
&& \hspace{-0.2in} =  
 ((\tvuN)_t,\vv)+\nu(\A\vu,\vv) + (\CC\tvu_N,\vv) 
     - (\vf,\vv)_M + (\B(\vu_N,\vu_N),\vv) \\
&& \hspace{-0.2in} = 
   ((\tvuN-\vu)_t,\vv)+(\vf,\vv) - (\vf,\vv)_M + (\CC\tvuN-\CC\vu,\vv) +
    (\B(\vu_N,\vu_N)-\B(\vu,\vu),\vv).  
\eeas
Using the orthogonal projection  $\proj_N$ in  (\ref{fin_space_vec}), 
we can write this relation in functional form as 
\[
\frac{d\vw_N}{dt} = (-\nu \A-\CC)\vw_N - \mathbf{\proj_N} \left[(\vp_N)_t + \CC\vp_N\right] 
+ \proj_N \vf - \dproj_N \vf 
   + \proj_N\left[\B(\vu_N,\vu_N) - \B(\vu,\vu)\right].
\]
Integrating with respect to $t$ and using  $\vw_N(0)=0$, we have 
\bea
 \vw_N (t) &=& 
  \int_{0}^t e^{-(t-s)(\nu \A+\CC)} \left[- \proj_N (\frac{d}{ds}\vp_N + \CC\vp_N) + \proj_N \vf - \dproj_N \vf \right](s) \;ds \nonumber \\
& &  + \int_{0}^t e^{-(t-s)(\nu \A+\CC)} \proj_N \left[\B(\vu_N,\vu_N)-\B(\vu,\vu)\right](s) \;ds. \label{wN-1}
\eea
Let
\begin{equation}\label{RN}
R_N(\epsilon, t-s)  = \|\nu^{\epsilon}\proj_N \A^{\epsilon} e^{-(t-s)(\nu \A+\CC)}\|.
\end{equation}
On using (\ref{weakLip}), with $\delta \in  (1/2,1)$,
and the uniform boundedness of the orthogonal projection $\proj_N$,  we get
\bea
& & \|e^{-(t-s)(\nu \A+\CC)}\proj_N\left[\B(\vu_N,\vu_N)-\B(\vu,\vu)\right]\| \\
 & & \le \frac{R_N(\delta, t-s)}{\nu^{\delta}} \|\A^{-\delta}\proj_N (\B(\vu_N,\vu_N)-\B(\vu,\vu))\| \nonumber
 \le C R_N(\delta,t-s)  \|\vu_N-\vu\|. \label{bd_1}
\eea
Taking norms and using 
$\|\vu_N -\vu\| \le \|\vw_N\|+\|\vp_N\|$ together with (\ref{RN}), 
and (\ref{bd_1}) in (\ref{wN-1}),  we obtain
\beas
\|\vw_N(t)\|   &\le &
\left[\|\frac{d}{dt}\vp_N\| + \|\CC\vp_N\|
+ \left\|\proj_N \vf - \dproj_N \vf\right\| \right]  
\int_{0}^{t} R_N(0, t-s)\;ds  \\        
&& + C \int_0^t R_N(\delta, t-s) (\|\vp_N(s)\|+\|\vw_N(s)\|) ds.
\eeas
Using Gronwall's inequality, we obtain for each $t \in [0,T]$,
\bea
\|\vw_N(t)\|  &\le & C \left\{
\left[\|\frac{d}{dt}\vp_N\| + \|\CC\vp_N\|
+  \left\|\proj_N \vf - \dproj_N \vf\right\| \right]  
\int_{0}^{t} R_N(0, t-s)\;ds   \right.  
\nonumber \\         && \qquad 
\left .
+ \|\vp_N\|\  \int_0^t R_N(\delta, t-s)\;ds \right \}. \label{Gron}
\eea
For $\epsilon \in [0,1]$, we have to bound 
\begin{equation}\label{Rmore}
 \int_0^t R_N(\epsilon , t-s)\;ds
=  \int_0^t R_N(\epsilon , r)\;dr.
\end{equation}
Using also (\ref{RN}) and \eqref{Coriolis_ZLm},
\[
 R_N(\epsilon, r) \leq \max_{1\le L\le N; |m|\le L } 
      |(\nu\lambda_L)^{\epsilon} e^{-\nu \lambda_L r}e^{-2\Omega irm\lambda^{-1/2}_L}|
    \le \max_{ z \in[\nu\lambda_1,\nu\lambda_N]} z^\epsilon e^{-rz}.
\]
Thus
\begin{equation} \label{Rcases}
 R_N(r) \le\begin{cases}
         (\nu \lambda_N)^\epsilon e^{-\nu\lambda_N r} & \mbox{ if } r \le \epsilon/(\nu\lambda_N),  \cr
         \epsilon^\epsilon e^{-\epsilon} r^{-\epsilon} & 
            \mbox{ if } \epsilon/(\nu\lambda_N) \le r \le \epsilon/(\nu\lambda_1)  ,\cr
        (\nu \lambda_1)^\epsilon e^{-\nu\lambda_1 r} & \mbox{ if } r \ge \epsilon/(\nu\lambda_1).\cr
     \end{cases}
\end{equation}
With  
\begin{equation} \label{sub_int}
 I_1 = [0,\epsilon/(\nu\lambda_N)]\cap[0,t], \quad
 I_2 = [\epsilon/(\nu\lambda_N),\epsilon/(\nu\lambda_1)]\cap [0,t], \quad
 I_3 = [\epsilon/(\nu\lambda_1),t]\cap [0,t],
\end{equation}
the interval of integration in (\ref{Rmore}) can be subdivided into
these three intervals. In particular, using (\ref{Rcases}) and (\ref{sub_int}),
we get 
\begin{equation} \label{I1}
  \int_{I_1} R_N(r) dr \le \int_{I_1} (\nu\lambda_N)^\epsilon e^{-\nu\lambda_N r} dr
    \le \frac{(1-e^{-\epsilon})}{(\nu\lambda_N)^{1-\epsilon}} \leq C,
\end{equation}
\begin{equation} \label{I2}
  \int_{I_2} R_N(r) dr \le \int_{I_2} \epsilon^\epsilon e^{-\epsilon} r^{-\epsilon} dr
  \le \frac{\epsilon e^{-\epsilon}}{(1-\epsilon)}
  \left[ \left(\frac{1}{\nu\lambda_1}\right)^{(1-\epsilon)} - 
         \left(\frac{1}{\nu\lambda_N}\right)^{(1-\epsilon)}\right] \le C,
\end{equation}
and
\begin{equation} \label{I3}
 \int_{I_3} R_N(r) dr \le \int_{I_3} (\nu\lambda_1)^\epsilon e^{\nu \lambda_1 r} dr
  \le \frac{ (e^{-\epsilon}-e^{-\nu\lambda_1 t}) }
           {(\nu \lambda_1)^{1-\epsilon}} \le C. 
\end{equation}
Using (\ref{I1})- (\ref{I3}) in (\ref{Rmore}), we get
\begin{equation}\label{bd}
\int_0^t R_N(\epsilon , t-s)\;ds \le C, \qquad \epsilon \in [0,1].
\end{equation}
Substituting this in (\ref{Gron}), we get 
\begin{equation}\label{penul}
\|\vw_N(t)\|  \le  C 
\left[\|\frac{d}{dt}\vp_N\| + \|\CC\vp_N\| +  \|\vp_N\|\  
+  \left\| \proj_N\vf - \dproj_N \vf\right\| \right].
\end{equation}
The term $\|\CC\vp_N \|$ in \eqref{penul} can be simplified 
using (\ref{C}) and the fact that $\vp_N$ is tangential 
(and hence $\xh \cdot \vp_N(\xh) = 0$), 
\[
\left[\xh \times \vp_N(\xh)\right] \cdot \left[\overline{(\xh \times \vp_N(\xh)}\right]  =  
\left[\xh \cdot \xh\right] \left[\vp_N(\xh) 
\cdot \overline{\vp_N(\xh)}\right] = \| \vp_N \|^2,
\]
and hence 
\begin{equation}\label{penul1}
\|\vw_N(t)\|  \le  C 
\left[\|\frac{d}{dt}\vp_N\| +  \|\vp_N\|\  
+    \left\|\proj_N \vf - \dproj_N \vf\right\| \right].
\end{equation}
Hence from Theorem~\ref{cont_est} and \ref{der_cont_est},
for all $t \in (0, T)$ we have
\begin{equation} \label{final_bound1}
  \|\vu-\vu_N\| \le   C \left[ N^{-2s-1} e^{-\mu(t) N}
+  \left\|\proj_N \vf - \dproj_N \vf\right\| \right].
\end{equation} 
In particular, using (\ref{app_power_vec}) 
and (\ref{disc_proj_prop2}) with $2s$ being an integer, we get
\begin{equation}  \label{final_bound2}
 \left\|\proj_N \vf - \dproj_N \vf\right\| 
\le 
  \left\|\proj_N\vf -  \vf\right\| +
  \left\| \vf - \dproj_N \vf\right\| \le C N^{-2s}.
\end{equation}
Now the result (\ref{int_bound}) follows from  (\ref{final_bound1}) 
and  (\ref{final_bound2}).
\end{proof}

\subsection{Adaptive and fast implementation of the  pseudospectral method}
Having established spectrally accurate convergence of the spatially discrete
scheme  \eqref{app1_weak_form}, for implementation of the scheme \eqref{app1_weak_form}
to simulate  stable and accurate solutions of  \eqref{NSE_sphere} and compare
with  benchmark random flow simulations in the literature, we need to 
discretize  the time derivative operator $\frac{d}{dt}$ 
in  \eqref{app1_weak_form}.
Further, at {\em each}
discrete time step  we develop a (FFT-based)  fast evaluation technique to
set up the resulting fully discrete nonlinear system with spatial 
$\calO(N^4)$ complexity.  First we consider discretization of $\frac{d}{dt}$ in 
\eqref{app1_weak_form}.

In order to the make  \eqref{app1_weak_form} fully discrete in space and time,
and hence compute the  $N^2+2N$  unknown  time-dependent coefficients
in the representation of the tangential divergence-free approximate real
velocity vector field
\begin{equation}\label{app_soln_rep}
\vu_N (\xh,t) := \sum_{L=1}^N \sum_{|m| \le L} \alpha_{L,m}(t) \Z_{L,m}(\xh), \quad
\alpha_{L,m} = \overline \alpha_{L,-m},
\quad   \alpha_{L,m}(0) = (\vu_0, \Z_{L,m})_M,  
\end{equation}
for $\xh \in S$ and $t \geq 0$,  the standard  fixed-time-step 
backward-Euler (or Crank-Nicolson) Galerkin approach could be used 
in (\ref{app_weak_form}), leading to a first-order (or a second-order, respectively) 
in time non-adaptive scheme~\cite{gan_kas_06}.  
However, due to the complicated unknown flow behavior of the NSE solutions,
when the initial  states are random, it is more efficient instead to integrate 
(\ref{app_weak_form}) using  a combination of  
multi-order integration formulas that allow 
adaptive choice of time step, leading to computation of solutions with a specified
accuracy in time. In this
paper we follow the latter approach.

For implementation purposes, we first need to  substitute (\ref{app_soln_rep}) 
in (\ref{app1_weak_form}). We write the resulting 
$N^2 + 2N$-dimensional system of ordinary differential equations (ODE), for the unknown 
$N^2 + 2N$ time dependent coefficients of $\vu_N(\xh,t)$  in \eqref{app_soln_rep},  
as
\begin{equation}\label{ode}
\frac{d}{dt}\boldsymbol{\alpha}(t) = \mathbf{F}(t,\boldsymbol{\alpha}(t)).
\end{equation}

It is well known that such nonlinear ODE systems  are stiff, and hence it 
is important to use time discretization techniques with a large stability region~\cite{har_wan_book}.
Further, it is important to use high-order implicit formulas whenever possible, but 
the high-order discretization formulas are appropriate only at those time steps 
where the unknown exact solution is smooth. 

For practical problems such as 
the Navier-Stokes equations (with initial random state)
where the necessary  spatial discretization (at each time step) is expensive, 
 it is important  to optimize computing time by simulating only
up to a required accuracy by choosing adaptive discrete time steps. 
Unlike the adaptive spatial mesh for 
elliptic PDEs based on the {\em a posteriori} estimates, 
adaptive time steps can be computed  by comparing 
numerical solutions obtained using two distinct order formulas~\cite{har_wan_book}
and hence simulation  using multi-order formulas is appropriate.

In particular, for practical realization of large stiff nonlinear ODE systems,
multi-order implicit backward differentiation formulas (BDF) and their generalizations 
such as the numerical differential formulas (NDF) are most appropriate.
The implicit NDF formula of order $p$ (NDFp) with a parameter $\kappa_p$ (so that $\kappa_p = 0$
corresponds to BDFp) for the system (\ref{ode}), with 
$\boldsymbol{\alpha}_n \approx  \boldsymbol{\alpha}(t_n)$ and $\nabla^m$ denoting the 
$m$-th Newton backward difference operator,  is
\begin{equation}\label{ode_disc}
\sum_{m = 1}^p \frac{1}{m} \nabla^m \boldsymbol{\alpha}_{n+1} =
h\mathbf{F}(t_{n+1},  \boldsymbol{\alpha}_{n+1}) + 
\kappa_p \nabla^{p+1} \boldsymbol{\alpha}_{n+1}  \sum_{j = 1}^p \frac{1}{j} .
\end{equation}
It is well known~\cite{har_wan_book} that BDFp (and hence NDFp) are unstable 
for $p > 6$, and for $p = 6$ the stability region is small and hence not practically
useful in our case. Further the celebrated Dahlquist barrier~\cite{har_wan_book} implies
that  BDFp (and hence NDFp) cannot be   absolutely stable 
[that is, $A(\alpha)$-stable with  $\alpha = 90\,^\circ$] for $p > 2$. 

Following details in~\cite{ode_suite}, for simulation of \eqref{ode} 
we use multi-order  NDFp with  $p = 1, 2, 3, 4, 5$ 
(and respective $\kappa_p = -0.1850, -1/9, -0.0823, -0.0415, 0$) 
and these are  $A(\alpha_p)$-stable, with respective
$\alpha_p = 90^\circ,  90^\circ, 80^\circ, 66^\circ, 51^\circ$. 
For $p = 1, 2, 3, 4, 5$, NDFp is more accurate than BDFp, however 
NDFp has slightly smaller stability angle compared BDFp only for $p= 3, 4$ 
(with respective $\alpha_p = 86^\circ,  73^\circ$) and the same stability
angle for $p = 1, 2, 5$.

For each fixed time  discretization step, the computational cost  
is dominated by evaluation of 
$\mathbf{F}(t_{n+1},  \boldsymbol{\alpha}_{n+1})$ in \eqref{ode_disc}
and it is important to have an efficient method to set up the 
spatial part of the  nonlinear system \eqref{app1_weak_form}.
Using the spectral properties of the Stokes operator $\A$ given by 
(\ref{stok_eig}),  the linear second term in   (\ref{app1_weak_form})
is trivial to evaluate using the diagonal matrix consisting of the
eigenvalues of $\A$. 

The Coriolis term  can be evaluated similarly
using the identity~\cite[Equation (24)]{cao_rammaha_titi99}
\begin{equation}\label{Coriolis_ZLm}
(\CC \; \Curl Y_{J,k}, \Z_{L,m}) =   
(2 \Omega \cos \theta \xh \times \Curl Y_{J,k}, \Z_{L,m}) =   
  - 2\Omega i \frac{m} {\lambda^{1/2}_L} \delta_{L,J} \delta_{k,m}.
\end{equation}
For the first component in the nonlinear third term in (\ref{app1_weak_form}),
we use (\ref{short_nonlin1}) to write
\bea 
 \B(\Z_{R,s},\Z_{J,k}) &=&  -\frac{1}{\sqrt{\lambda_R
     \lambda_J}}\PP_{\Curl}(\Delta Y_{R,s} \Grad Y_{J,k}), \label{eval_nonlin}
 \\
 \left(\B(\Z_{R,s},\Z_{J,k}),\Z_{L,m}\right) &=&  \sqrt{\frac{\lambda_R}
{\lambda_J\lambda_L }}\left(Y_{R,s} \Grad Y_{J,k}, \xh \times  \Grad Y_{L,m}\right).\label{eval_nonlin1}
\eea
It is convenient to write $\Grad Y_{J,k}$ and  $\xh \times  \Grad Y_{L,m}$ in terms
of expressions similar to those in (\ref{sph_har}). Such  explicit representations
are also useful for the efficient evaluation of the $N^2$ Fourier coefficients $(\vf,\Z_{L,m})_M$ 
of the source term in (\ref{app1_weak_form}), and eventually for the computation of
the vorticity field.

In order to express the tangential (and normal, needed for computing the approximate 
vorticity from $\vu_N$) vector harmonics
as a linear combination of the scalar harmonics  (\ref{sph_har}), we first recall,  from the classical  quantum mechanics
literature (see, for example, \cite{varshalovich}), the covariant spherical basis vectors 
\be \label{covar}
  \ve_{+1} = - \frac{1}{\sqrt{2}} ([1,0,0]^T + i [0,1,0]^T), \quad 
  \ve_{0} = [0,0,1]^T, \quad 
  \ve_{-1} = \frac{1}{\sqrt{2}}([1,0,0]^T - i [0,1,0]^T),
\ee
and the Clebsch-Gordan coefficients  
\be
C^{j,m}_{j_1,m_1,j_2,m_2} := 
(-1)^{(m+j_1-j_2)}\sqrt{2j+1}\left(\begin{array}{lll}j_1 & j_2 & j \\ m_1 &  m_2 & -m
\end{array} \right),
\ee
where $\left(\begin{array}{lll}a & b & c \\ \alpha &  \beta & \gamma
\end{array} \right)$ are the  Wigner 3j-symbols given, for example, by the Racah formula, 
\beas
& & \hspace{-0.5in}
\left(\begin{array}{lll}a & b & c \\ \alpha &  \beta & \gamma
\end{array} \right) \\& & \hspace{-0.5in} = 
(-1)^{(a-b-\gamma)}\sqrt{T(abc)}\sqrt{(a+\alpha)!(a-\alpha)!(b+\beta)!(b-\beta)!(c+\gamma)!(c-\gamma)!} \times \\ & & \sum_t\frac{(-1)^t}
{t!(c-b+t+\alpha)!(c-a+t-\beta)!(a+b-c-t)!(a-t-\alpha)!(b-t+\beta)!},  
\eeas
where the sum is over all integers $t$ for which the factorials in the denominator 
all have nonnegative arguments. In particular, the number of terms in the sum
is $1+ \min\{a\pm \alpha, b\pm\beta, c\pm \gamma, a+b-c, b+c-a, c+a-b\}$. 
The triangle coefficient $T(abc)$ is  defined by 
\[
 T(abc) = \left[\frac{(a+b-c)!(a-b+c)!(-a+b+c)!}{(a+b+c+1)!} \right].
\]
Below, we require $C^{j,m}_{j_1,m_1,j_2,m_2}$ only for some $j_2, m_2 \in \{-1, 0, 1\}$, and 
using various symmetry and other known properties (such as 
$C^{j,m}_{j_1,m_1,j_2,m_2} = 0$ unless the conditions $|j_1 - j_2 | \le j \le j_1 + j_2$
and $m_1 + m_2 = m$  hold) of Wigner 3j-symbols, these coefficients can be efficiently 
pre-computed and stored.

In our computation, we used the following basis functions
for the tangential vector fields: 
$(i)~ \Grad Y_{L,m}, (ii)~ \xh \times \Grad Y_{J,k}$. 
For the vorticity  components of $\vu_N$, 
in addition we used  
$(iii)~ \Vort \Z_{J,m} =  \Curln\times \Z_{J,m} = \lambda_J^{1/2} \xh Y_{J,m}
= -\lambda_J^{-1/2} \xh \Delta  Y_{J,m} $. In particular, using (\ref{app_soln_rep})
our approximation  to the vorticity in (\ref{eq:vort}), for  a fixed $t \geq
0$ and  $\xh \in S$, is 
\begin{equation}\label{eq:vort_app}
 \Vort \vu_N (\xh,t)  = \Curln \vu_N (\xh) = \xh \Delta \Psi_N(\xh),
\end{equation}
where
\[
\Psi_N(\xh,t) = -\sum_{L=1}^N \sum_{|m| \le L} \lambda_L^{-1/2} \alpha_{L,m}(t)  Y_{L,m}(\xh).
\]

To facilitate  easy application of  fast transforms to evaluate 
these functions at the  $M = \mathcal{O}(N^2)$ quadrature points 
$\{\hxi_{p,q} = \p(\theta_p,\phi_q)\}$, we represent 
these three types of fields first as a linear combination 
of the covariant vectors in (\ref{covar}): 
\bea
\Grad Y_{L,m} &=& B_{+1,L,m} \ve_{+1} +
                  B_{0,L,m} \ve_{0} +
                  B_{-1,L,m} \ve_{-1}, \label{grady_term}\\
(\xh \times \Grad Y_{J,k})
              &=&  D_{+1,J,k} \ve_{+1} +
                   D_{ 0,J,k} \ve_{0}  +
                   D_{-1,J,k} \ve_{-1}, \label{curly_term}
\eea
With  $c_L =  (L+1) \sqrt{\frac{L}{2L+1}}, d_L =  L \sqrt{\frac{L+1}{2L+1}}$, 
these coefficients are explicitly given by 
\beas
 B_{+1,L,m} &=&
   \left\{c_L C^{L,m}_{L-1,m-1,1,1} P^{m-1}_{L-1}(\cos \theta) 
 + d_L C^{L,m}_{L+1,m-1,1,1} P^{m-1}_{L+1}(\cos \theta)\right\} e^{i(m-1)\varphi}\\
B_{0,L,m} &=&
   \left\{c_L C^{L,m}_{L-1,m,1,0} P^{m}_{L-1}(\cos \theta)
 + d_L C^{L,m}_{L+1,m,1,0} P^{m}_{L+1}(\cos \theta)\right\} e^{im\varphi}\\
B_{-1,L,m} &=&
  \left\{c_L C^{L,m}_{L-1,m+1,1,-1} P^{m+1}_{L-1}(\cos \theta)  
 + d_L C^{L,m}_{L+1,m+1,1,-1} P^{m+1}_{L+1}(\cos \theta)\right\} e^{i(m+1)\varphi}.
\eeas
\beas
 D_{+1,J,k} &=& i\sqrt{\lambda_J}C^{J,k}_{J,k-1,1,1} P^{k-1}_J(\cos\theta) e^{i (k-1) \varphi}, \\
 D_{0,J,k} &=& i\sqrt{\lambda_J} C^{J,k}_{J,k,1,0}P^{k}_J(\cos\theta) e^{i k \varphi},\\
 D_{-1,J,k}&=&i\sqrt{\lambda_J} C^{J,k}_{J,k+1,1,-1}P^{k+1}_J(\cos\theta) e^{i(k+1)\varphi}.
\eeas

Noting
(i)  the complex azimuthal exponential terms $e^{ik\varphi}, e^{im\varphi}$ in
(\ref{sph_har})
and  (\ref{grady_term})-(\ref{curly_term}) (via the above representations
for $B$ and $D$) for $|k| \leq J, |m| \leq L; 1 \leq L,J \leq N$,  and
(ii) the need to evaluate {\em several $\mathcal{O}(N^2)$} sums, of the form in
(\ref{app_soln_rep})
and  (\ref{eval_nonlin})-(\ref{eval_nonlin}),  
at the equally spaced $\mathcal{O}(N)$ azimuthal  quadrature points (see (\ref{quadrule})), 
we reduce the complexity by $\mathcal{O}(N)$ in {\em each} of these sums, at
{\em each} adaptive-time step (described below),  by using the FFT for setting up 
the nonlinear system \eqref{app1_weak_form}, similar to the approach in~\cite{math_bio_pap}.
In our numerical experiments (see Section~\ref{num_exp}) for  
adaptive-time simulation  of a flow induced by random initial states,   
we observed that such  an efficient FFT based implementation 
reduced the (non-FFT code)  computing time  substantially for the case 
$N = 100$, to simulate from $t = 0$ to $t = 60$.

In addition,  by using the fast Legendre/spherical
transforms along the latitudinal direction
(obtained, for example, by modifying the NFFT algorithm 
in  \cite{fast_sph} for  evaluation of the Legendre 
functions in the above terms 
at  $\mathcal{O}(N)$ non-uniform latitudinal quadrature  points),   
we could reduce the complexity by $\mathcal{O}(N^2)$.
We did not use the fast Legendre/spherical transforms
in our implementation due to
the spectral convergence of the scheme 
and the fact that   $N \leq 100$ in our simulations. 
(In these complexity counts, we  ignored $\mathcal{O}(\log N)$ 
and $\mathcal{O}(\log^2 N)$ terms.)

\section{Numerical Experiments}\label{num_exp}
We demonstrate the fully discrete pseudospectral quadrature algorithm 
by simulating (i) a known solution test case with low to high 
frequency modes and (ii) a benchmark
example~\cite[page~305]{debussche_dubois_temam}  
in which the unknown velocity and vorticity fields 
are generated by a random initial state.

The first test example is useful to demonstrate that the 
pseudospectral quadrature algorithm reproduces any number of high
frequency modes in the solution (provided $V_N$ contains all these modes),
with computational error dominated only by the chosen accuracy 
for the adaptive time evolution for the ordinary differential system
(\ref{ode}). 

\subsection{Example~1.}
Our test case first example is (\ref{NSE_sphere}) with
\begin{equation}\label{ini_state}
 \vu|_{t=0}(\xh) = \vu_0(\xh) = g(0)[\W_1(\xh) - \W_2(\xh)],
\end{equation}
\begin{equation}\label{g_for_ini}
g(t) = \nu e^{-t}[\sin(5t) + \cos(10t)],  \quad  
\W_1 = \Z_{1,0} + 2\Re(\Z_{1,1}), \; 
\W_2 = \Z_{2,0} + 2\Re(\Z_{2,1} + \Z_{2,2}),
\end{equation}
where $\Z_{L,m}$ is given by (\ref{stok_eig}), $\Re(\cdot)$ denotes the
real-part function, and the external force $\vf(\xh,t)$ in (2.1) 
is chosen so that 
\begin{equation}\label{exact_soln}
\vu(\xh,t) =  tg(t)\sum_{L=1}^{N_0}
  \left[{\bf Z}_{L,0}+ 2 \sum_{m=1}^L  \Re(\Z_{L,m})\right](\xh)
        + g(t) \W_1(\xh)  + (t-1)g(t)\W_2(\xh),
\end{equation}
is the exact {\em tangential} divergence-free velocity field,
solving the NSE (\ref{NSE_sphere}).
The exact test field (\ref{g_for_ini})-(\ref{exact_soln}) has 
high oscillations both in space and time, and exponentially decays 
in time. Note the dependence on a parameter $N_0$, the maximum order
of the spherical harmonics in the exact solution.

In our calculation of the approximate solution $\vu_N$, we chose
$N=N_0$, so that all frequencies of the exact solution can be
recovered. The solution (\ref{exact_soln}) 
is then used to validate our algorithm and code 
by numerical adaptive time-integration of (\ref{app_weak_form}), 
for various values of $N=N_0$. 
In particular, for a fixed integration tolerance error, 
we demonstrate in Figure~\ref{exact_soln_err_fig}
that all $N$ modes in \eqref{exact_soln} can indeed be recovered 
by the approximate solution $\vu_N$,  
within the chosen error tolerance, 
for all $N = N_0 = 70, 80, 90, 100$.
\begin{figure}[h]
\centering
\includegraphics[width=12cm,height=9cm]{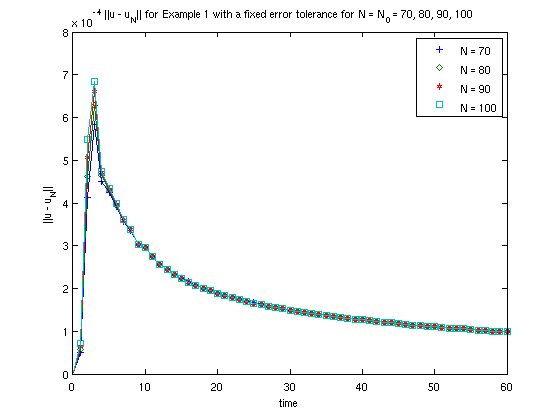}
\caption{$\|\vu - \vu_N\|$ for Example 1 with a fixed time-integration error, 
$N = N_0 = 70, 80, 90, 100$. }\label{exact_soln_err_fig}
\end{figure}

\subsection{Example 2.}
Having established the validity of our algorithm for a simple 
known solution, 
we use the same code to simulate unknown velocity and vorticity 
fields generated
by (\ref{NSE_sphere}) with random initial velocity
field as in~\cite{debussche_dubois_temam, fengler_freeden} with  
angular velocity of the rotation 
$\Omega = 1$  and  $\nu^{-1} = 10,000$  in (\ref{NSE_sphere}).

The initial flow and external force in this benchmark example
satisfy the main  assumption (\ref{main_assumption}) for 
{\em any}  $s, \sigma_1>0$ and hence, as proved in Theorem~\ref{conv_anal},
the approximate solution $\vu_N$ is spectrally accurate and converges
super-algebraically with order given by (\ref{int_bound}) for {\em any} 
$s > 0$. As mentioned in Section~\ref{introduction}, this is the main advantage of the present paper over  the recent paper~\cite{fengler_freeden}, 
where such spectrally accurate convergence results are neither discussed 
nor proved. On the other hand, convergence results 
for two-dimensional problems on a Euclidean plane, 
supported by numerical experiments, 
have formed a core part of research on 
the NSE over the last few decades, 
see~\cite{debussche_dubois_temam,Foias_book, temam79} and references therein.

The random initial tangential divergence-free velocity field, 
having properties similar to those considered 
in~\cite[page~305]{debussche_dubois_temam}
and~\cite[page~988]{fengler_freeden}
(but not exactly same as  in~\cite{debussche_dubois_temam, fengler_freeden}, 
due to randomness), is a smooth function $\vu_0 = \vv \in G^{s/2}_\sigma$, 
a Gevrey class of order $s$ and  index $\sigma$, see (\ref{Gev_space_char}),
for {\em any} $s, \sigma > 0$, with Fourier 
coefficients $\widehat{\vv}_{L,m},~~ L = 1, 2, \cdots,  |m| \leq L$,  
defined by
\begin{equation}
\widehat{\vv}_{L,m} = 
 \begin{cases} 
a_L  \exp(i \phi_m)   &   L=1,\ldots,20;~
  m=0,\ldots,L, \cr
a_L (-1)^m   \exp(-i \phi_m)   &   L=1,\ldots,20;~
  m=-L, \cdots, -1, \cr
0, & L > 20;~ m = -L, \ldots, L,
\end{cases}
\end{equation}
where  $\phi_0 = 0$, and  $\phi_m \in (0,2\pi)$  are random numbers 
for $m > 0$, 
and $a_L = b_L/||\mathbf{b}||$, with $\mathbf{b} \in \R^{20}$ 
having components  $b_L = 2/\left[L+(\nu L)^{2.5}\right],~L=1,\ldots, 20$.
The vorticity stream function $\Psi$ (see (\ref{eq:vort})) of 
$\Vort \vu(\cdot,0)$ in Figure~\ref{vort_t0} demonstrates the randomness
of the field at time  $t = 0$. 

The external force field $\vf = \vf(\xh,t)$ in
(\ref{NSE_sphere}) for our simulation is motivated by that considered
in~\cite[page~305]{debussche_dubois_temam} and is exactly same as that in 
~\cite[page~988]{fengler_freeden}. The source term $\vf$ is a 
decaying tangential divergence-free field which, 
for {\em any} $s, \sigma > 0$, belongs to 
$C([0,T];\calD(\A^{s+1/2} e^{\sigma \A^{1/2}}))$ (for any $T > 0$) 
with the only  non-zero Fourier coefficient being 
$\widehat{\vf(t)}_{3,0}$. The Fourier
coefficient $\widehat{\vf(t)}_{3,0}$ is a continuous function
in time and  is defined by 
\[
  \widehat{\vf(t)}_{3,0}  = \begin{cases} 
                 1, &  0 \le t \le 10, \cr
                 \cos(\pi t/5)\exp(-(t-10)/5) & t>10. \cr
              \end{cases}
\] 
The impact of the external force on the numerical
velocity and its complement (generating the approximate 
inertial manifold) is demonstrated in Figure~\ref{velocity_force_fig}
with the time evolution of the velocity field matching that of
the external force, leading to little change in evolution process
of the velocity 
as the external force gets smaller and smaller.

We chose a fixed relative error 
tolerance to be of accuracy at least $\mathcal{O}(10^{-3})$, 
for adaptive time-integration solving the $N^2+2N$-dimensional system   
ordinary differential equations  (\ref{app_weak_form}) in time,  
using  backward differential formulas with variable order (one to five) 
and variable adaptive time step sizes that meet the fixed error tolerance. 

As discussed in the next subsection, the high-frequency 
components of the solution turn out to be
unreliable as $t$ increases, \cite{debussche_dubois_temam},
presumably because  of the time discretization error. We therefore
retained only the frequency components up to some order $N_1 \le N$,
correspondingly, we define
an additional approximation of $\vu_N$, defined in 
(\ref{app_weak_form})-(\ref{app_soln_rep}), by 
\begin{equation}\label{eq: cut_app_soln} 
\vu_{N_1;N} (\cdot,t) :=   \proj_{N_1} \vu_N (\cdot,t), \qquad N_1 \leq N.
\end{equation}
As in Theorem~\ref{conv_anal}, for the spatially discrete  
pseudospectral quadrature method, assuming
\eqref{main_assumption} and exact time integration 
of (\ref{app_weak_form}), and hence using \eqref{bounded_soln},
\eqref{app_power_veca},
and \eqref{int_bound}, we get spectral convergence: 
\[
  \|\vu- \vu_{N_1;N}\| 
\le 
  \left\|\vu - \proj_{N_1} \vu\right\| +  
\left\|\proj_{N_1}\left(\vu -   \vu_N\right)\right\| \le
    C \left[N_1^{-2(s+1)} +  N^{-2s} \right] \leq C N_1^{-2s}.
\]
Our simulated approximate velocity fields 
\[
\vU_{75}(t) := \vu_{75;100}(t), 
 \quad \mbox{ for }t = 10, 20, 30, 40, 50, 60,
\] 
are in Figure~\ref{panel_t10}--\ref{panel_t60} and the
associated vorticity stream function $\Psi_{75}(t)$ of
$\Vort \vU_{75}(t)$ (computed using (\ref{eq:vort_app}))  
are in Figure~\ref{vort_t10}--\ref{vort_t60}.
These figures demonstrate that the initial random flow
with several smaller structures evolve into regular flow with
larger structures, similar to those observed 
in~\cite[page~307]{debussche_dubois_temam}. The choice
of $N_1$ will be discussed in the next subsection.
\subsection{Energy spectrum of the solution} 
\label{energy_spec_discussion}

If  $\vu$ (and hence the number of modes in $\vu$) is unknown, 
as it is in the case in the benchmark test Example~2, 
and if $\vu_N$ does not contain
all of the modes in $\vu$, the higher modes of $\vu_N$ are
usually less accurate than the chosen practical error tolerance and
hence can  even violate important physical properties of $\vu$
(because of the error time-integration tolerance being not
{\em very} small). In such cases, it is important to choose  
$N_1 < N$, depending
on certain known physical properties of  $\vu$. 

For a fixed time $t$, the $L$-th mode energy spectrum of a 
tangential divergence-free flow $\vu$ on the sphere  is defined by
\begin{equation}\label{energy}
E(L) = E(\vu,L) =   \sum_{|m| \le L} |\beta_{L,m}(t)|^2, \qquad  \qquad 
\vu(\xh,t) := \sum_{L=1}^\infty \sum_{|m| \le L} \beta_{L,m}(t) \Z_{L,m}(\xh).
\end{equation}
Although the analytical form of the flow in Example~2 is not known, 
several investigations have been carried out for such fields with 
initial spectrum of the $L$-th mode decaying with order $L^{-1}$ or $L^{-2}$. 
In particular, it is well known (see \cite{debussche_dubois_temam}),
for this benchmark test case (on periodic two dimensional geometries), 
that the energy spectrum of the velocity has a power-law inertial range 
and an exponential decay (dissipation range) for wave numbers larger than
the Kraichnan's dissipation wave number. Further, several random smaller 
structures built in the initial random vorticity evolve into regular flow 
with larger structures. 

With $\vu$ being the unique solution of the NSE (\ref{NSE_sphere}),
let us decompose $\vu = \tilde{\vu}_{N_1} + \vw_{N_1}$, where 
$\tilde{\vu}_{N_1}=\proj_{N_1}(\vu)$ contains all modes lower or equal $N_1$
and $\vw_{N_1} := \vu - \proj_{N_1}(\vu)$ contains all higher modes.
The existence of a relation between $\vw_{N_1}$ and $\tilde{\vu}_{N_1}$
of a form $\vw_{N_1} = \Phi(\tilde{\vu}_{N_1})$ was established
in ~\cite{temam_wang}. The graph of $\Phi$ is known as the
inertial manifold of \eqref{NSE_sphere}. For computational purposes,
the higher modes can be computed efficiently using an 
approximate inertial manifold: 
\begin{equation}\label{eq:p-p}
\tilde{\Phi}_{\widetilde {N_1}}( \tilde{\vu}_{N_1}) = 
(\nu \A +\CC)^{-1} 
\left(\proj_{\widetilde {N_1}}  - \proj_{N_1}\right) 
\left[\vf - 
\B\left( \tilde{\vu}_{N_1}, \tilde{\vu}_{N_1}\right)\right], 
~~ \widetilde{N_1} > N_1.
\end{equation}
This well known approximation (without the Coriolis term)
was  introduced in~\cite{foias_jolly_titi,foias_manley_temam} for
nonlinear dissipative systems, including the NSE on domains and
we choose $\widetilde{N_1} = 2N_1$.

For the benchmark test case, in the dissipation range, even
$L^4 E(\tilde{\Phi}_{2N_1}(\tilde{\vu}_{N_1}),L)$ decays exponentially 
to zero, further justifying restriction of the infinite dimensional 
range after certain values of $N_1$. As  
discussed in~\cite[page~280, 307]{debussche_dubois_temam}
(and repeated in~\cite[page~991]{fengler_freeden}), 
such a faster decay (one order higher than the $L^{-3}$ decay known is
Kraichnan theory of turbulence) is expected due to the Reynolds number
considered in~\cite{debussche_dubois_temam, fengler_freeden} being much 
small than $25,000$. (Extensive study in~\cite{Brachet, Orszag} shows that the
turbulence theory  decay can occur only when the Reynolds number is of the
order $25,000$.)

Briefly (without repeating technical details
in~\cite{debussche_dubois_temam}), using  the viscosity term $\nu$ in
(\ref{NSE_sphere}), the Reynolds number is $\mathcal{O}(\nu^{-1})$ 
with the order constant $r$ given by the product of  
the mean fluid velocity and the characteristic length-scale. With 
$\nu^{-1} = 10,000$ and the constant rotation rate $\Omega = 1$, 
for the Coriolis parameter in (\ref{C}), the  decay  of energy spectrum of
a numerical velocity and exponential decay of  its approximate complement  in 
Figure~\ref{energy_spectrum}-\ref{scaled_energy_spectrum} 
is well supported by  extensive
simulations in~\cite{Brachet, Orszag, debussche_dubois_temam}, 
highlighting further the benchmark applicability of our algorithm, 
extending periodic domain 
results in~\cite{Brachet, Orszag, debussche_dubois_temam}
to a practically relevant rotating sphere case with Coriolis effect. 

Further, the simulated
results are substantiated by the well known exponential decay of 
$E(\tilde{\Phi}_{2N_1}(\tilde{\vu}_{N_1}),L)$, observed in 
Figures~\ref{energy_spectrum}-\ref{scaled_energy_spectrum}.
Finally, the exponential decay of the energy spectrum 
show that $N = 100, N_1 = 3N/4$ is sufficient to understand 
the flow behavior for this benchmark example with our algorithm
using adaptive variable order and  variable time-step
highly stable  backward differentiation formulas with a
practically useful relative
error tolerance  $\mathcal{O}(10^{-3})$.


\begin{center}
\begin{figure}[h]
\includegraphics[width=12cm,height=7cm]{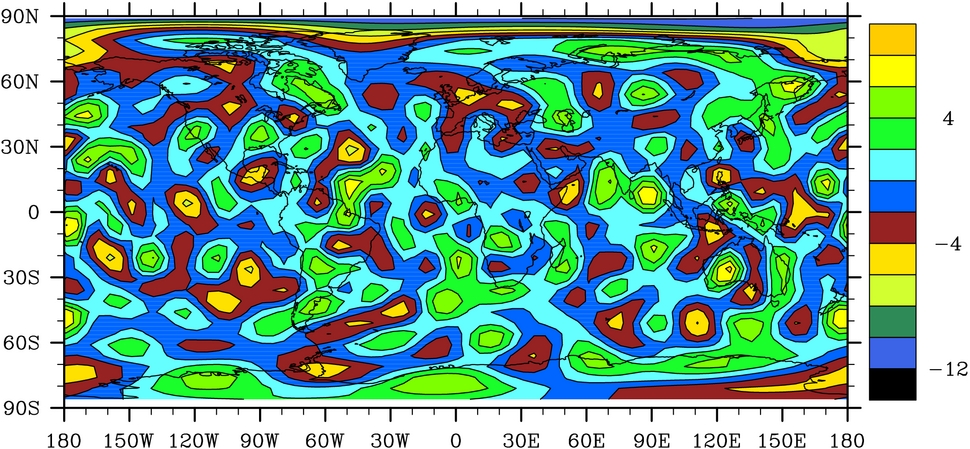}
\vspace{-0.1in}
\caption{Initial random vorticity stream function $\Psi$ of $\Vort \vu$ at $t = 0$.}\label{vort_t0}

\vspace{0.4in}
\begin{tabular}{cc}
\includegraphics[width=7.5cm,height=7.5cm]{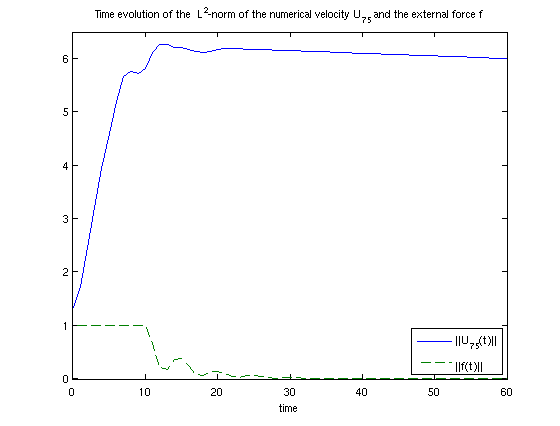}
& \hspace{-0.5in}
\includegraphics[width=7.5cm,height=7.5cm]{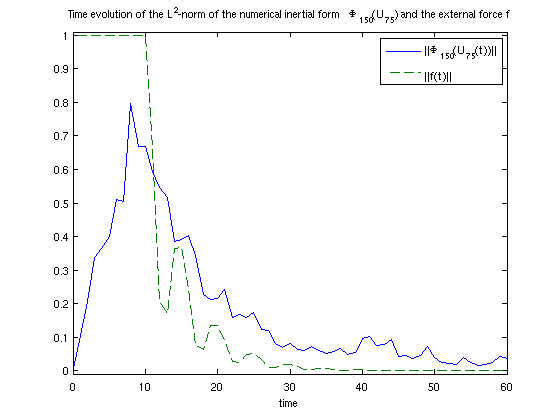}\\
Time evolution of  $\vU_{75}(t)$ and $f(t)$.  &  \hspace{-0.5in} Time evolution of 
$\tilde{\Phi}_{150}(\vU_{75}(t))$ and $\vf(t)$.
\end{tabular}
\caption{Impact of the  external force on a
numerical velocity and its complement.}\label{velocity_force_fig}
\end{figure}
\end{center}


\begin{center}
\begin{figure}[h]
\includegraphics[width=15cm,height=8cm]{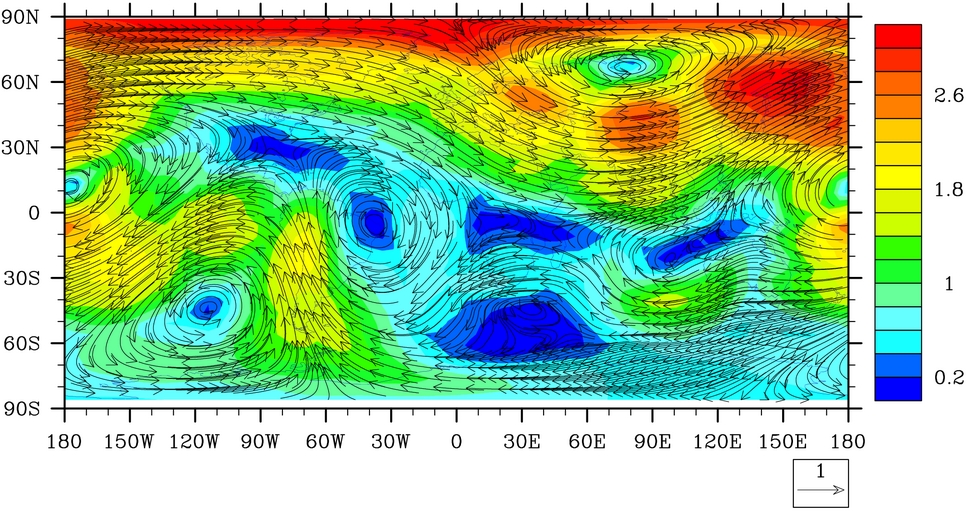}
\vspace{-0.6in}
\caption{Numerical velocity  $\vU_{75}(t)$,  at $t = 10$.}\label{panel_t10}
\includegraphics[width=15cm,height=8cm]{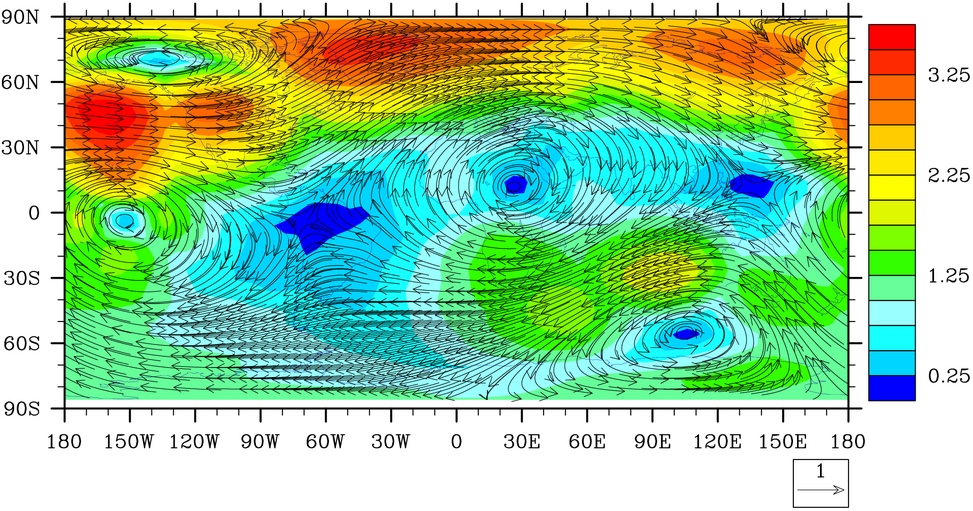}
\vspace{-0.6in}
\caption{Numerical velocity  $\vU_{75}(t)$,  at $t = 20$.}\label{panel_t20}
\includegraphics[width=15cm,height=8cm]{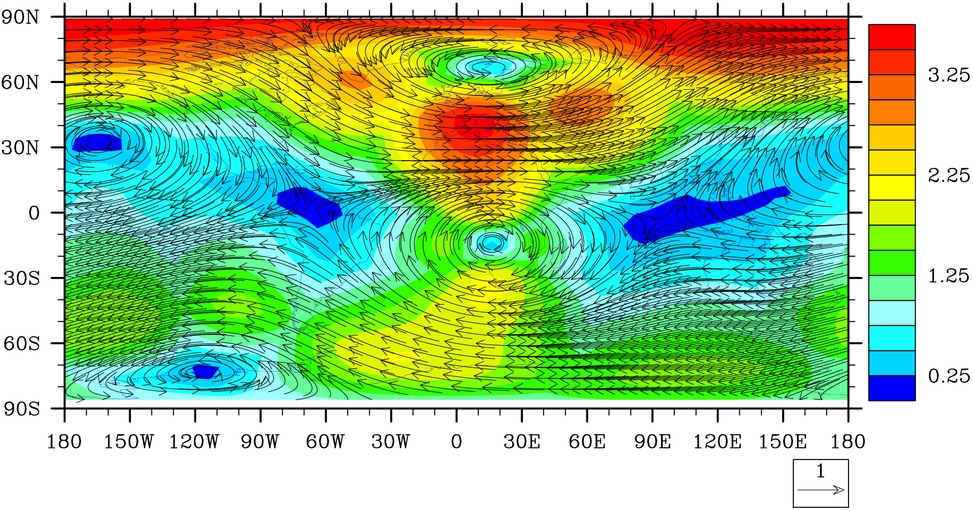}
\vspace{-0.6in}
\caption{Numerical velocity  $\vU_{75}(t)$,  at $t = 30$.}\label{panel_t30}
\end{figure}
\end{center}
\begin{center}
\begin{figure}[h]
\includegraphics[width=15cm,height=8cm]{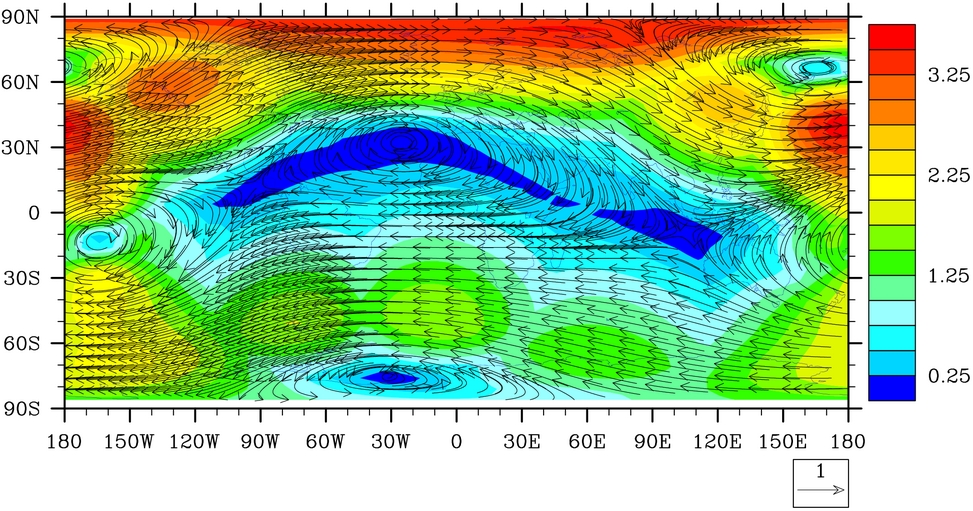}
\vspace{-0.6in}
\caption{Numerical velocity  $\vU_{75}(t)$,  at $t = 40$.}\label{panel_t40}
\includegraphics[width=15cm,height=8cm]{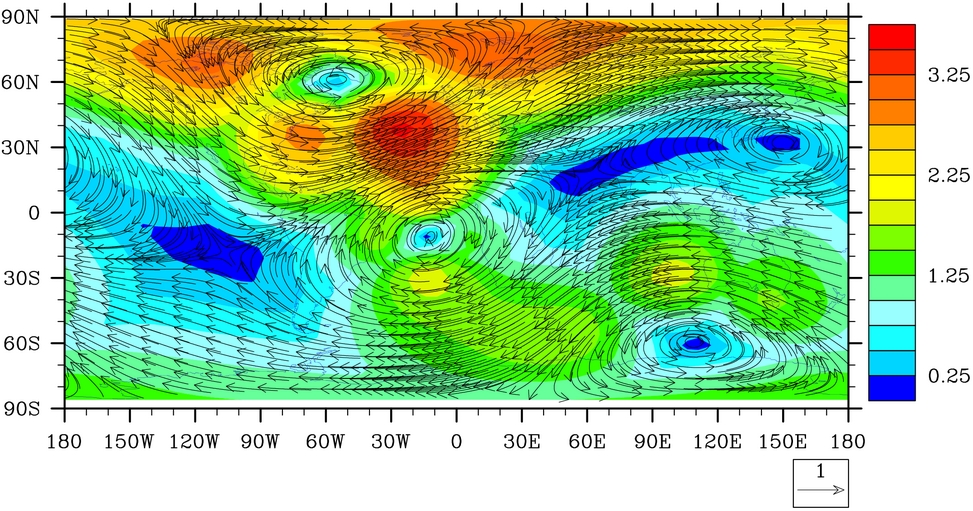}
\vspace{-0.6in}
\caption{Numerical velocity  $\vU_{75}(t)$,  at $t = 50$.}\label{panel_t50}
\includegraphics[width=15cm,height=8cm]{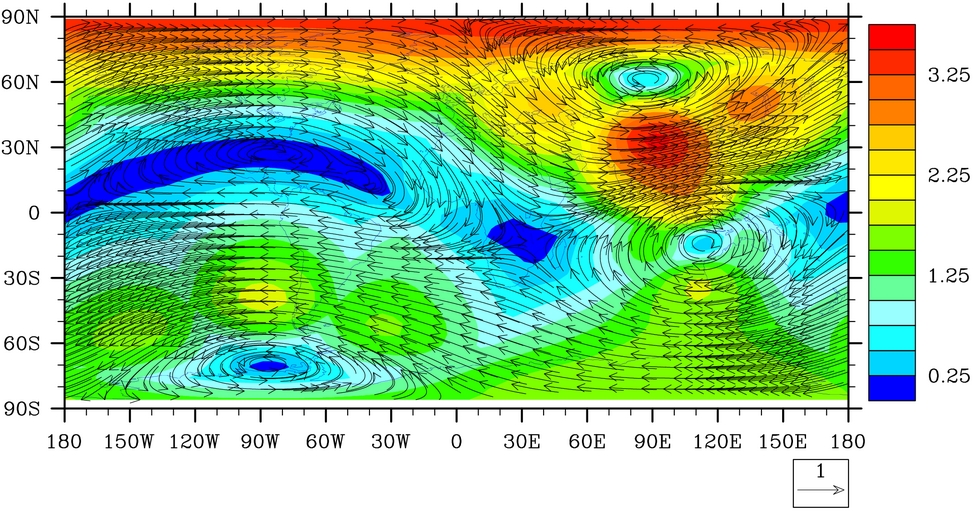}
\vspace{-0.6in}
\caption{Numerical velocity  $\vU_{75}(t)$,  at $t = 60$.}\label{panel_t60}
\end{figure}
\end{center}

\begin{center}
\begin{figure}[h]
\includegraphics[width=15cm,height=7cm]{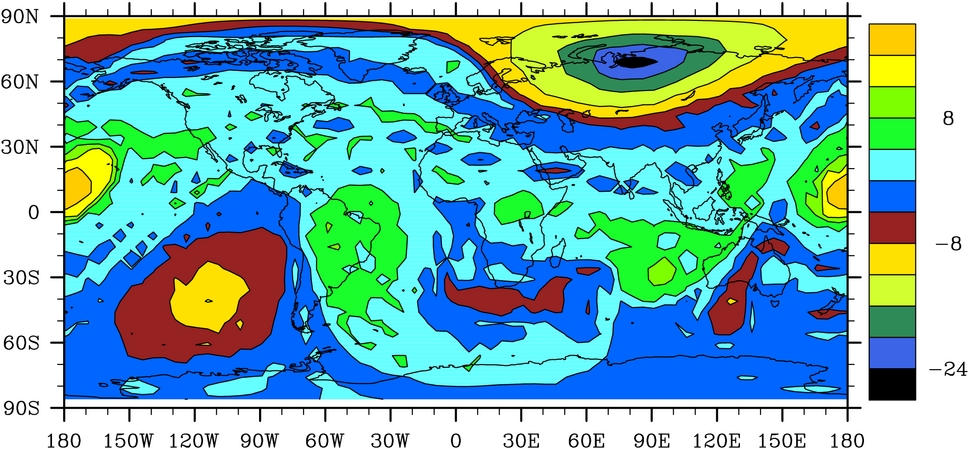}
\vspace{-0.3in}
\caption{Vorticity stream function $\Psi_{75}$ of $\Vort \vU_{75}$ at $t = 10$.}\label{vort_t10}
\includegraphics[width=15cm,height=7cm]{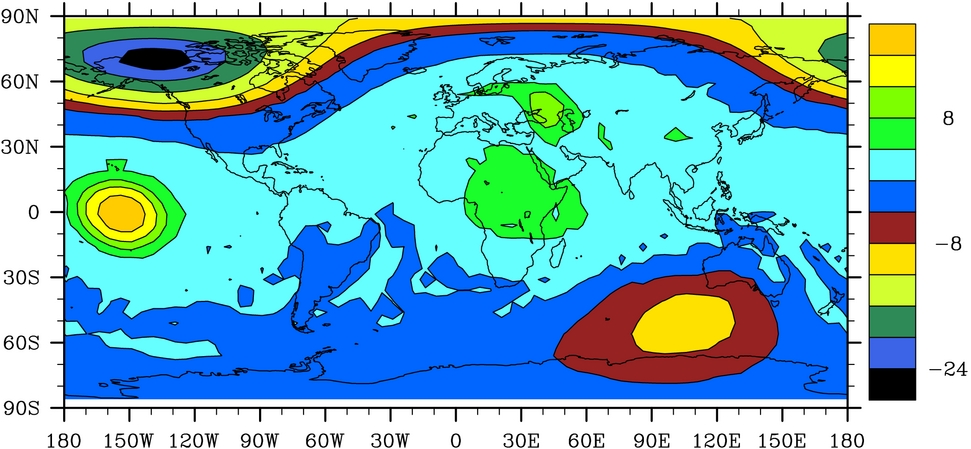}
\vspace{-0.3in}
\caption{Vorticity stream function $\Psi_{75}$ of $\Vort \vU_{75}$ at $t = 20$.}\label{vort_t20}
\includegraphics[width=15cm,height=7cm]{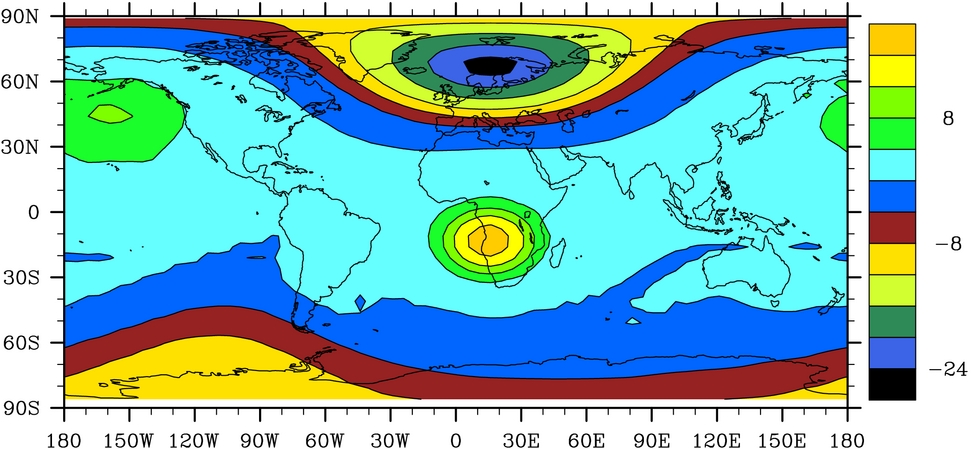}
\vspace{-0.3in}
\caption{Vorticity stream function $\Psi_{75}$ of $\Vort \vU_{75}$ at $t = 30$.}\label{vort_t30}
\end{figure}
\end{center}
\begin{center}
\begin{figure}[h]
\includegraphics[width=15cm,height=7cm]{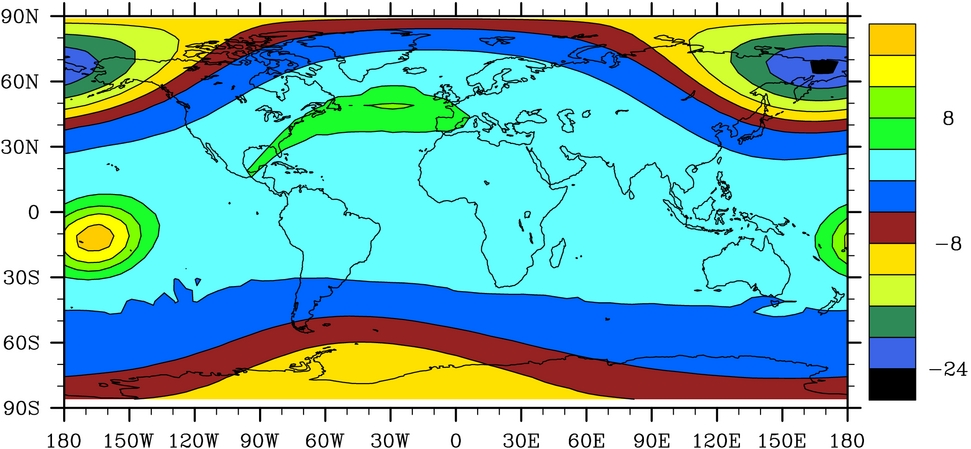}
\vspace{-0.3in}
\caption{Vorticity stream function $\Psi_{75}$ of $\Vort \vU_{75}$ at $t = 40$.}\label{vort_t40}
\includegraphics[width=15cm,height=7cm]{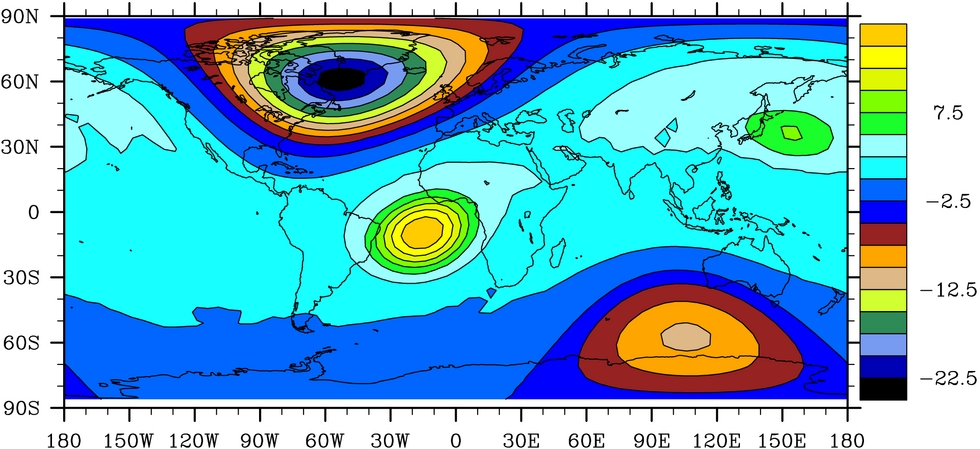}
\vspace{-0.3in}
\caption{Vorticity stream function $\Psi_{75}$ of $\Vort \vU_{75}$
  at $t = 50$.}\label{vort_t50}
\includegraphics[width=15cm,height=7cm]{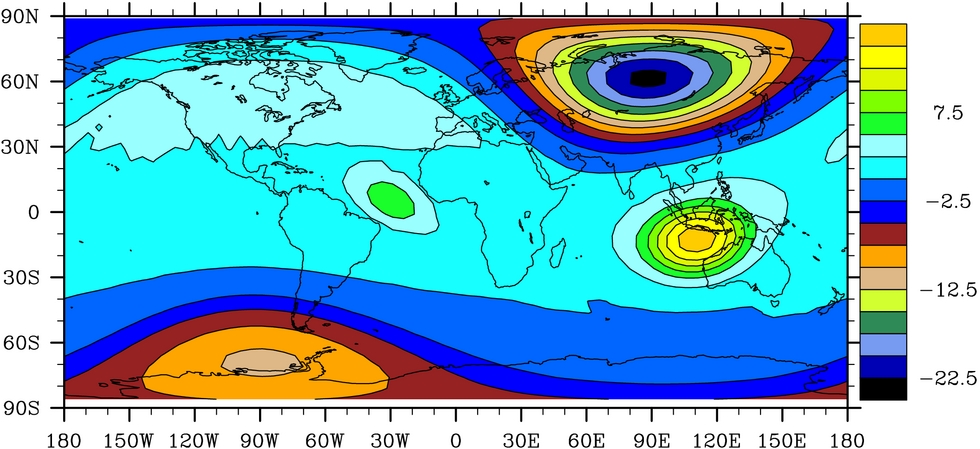}
\vspace{-0.3in}
\caption{Vorticity  stream function $\Psi_{75}$ of $\Vort \vU_{75}$ at $t = 60$.}\label{vort_t60}
\end{figure}
\end{center}

\begin{center}
\begin{figure}[h]
\begin{tabular}{cc}
\includegraphics[width=7.5cm,height=7.5cm]{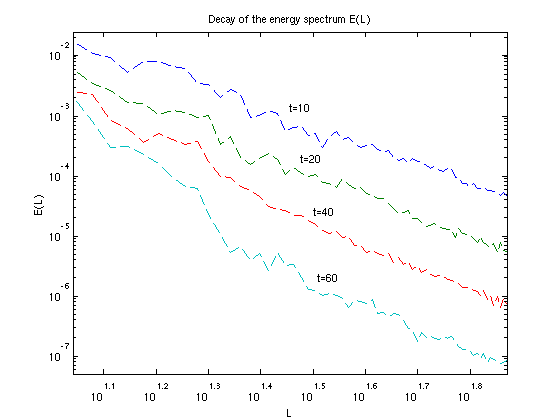}
&
\includegraphics[width=7.5cm,height=7.5cm]{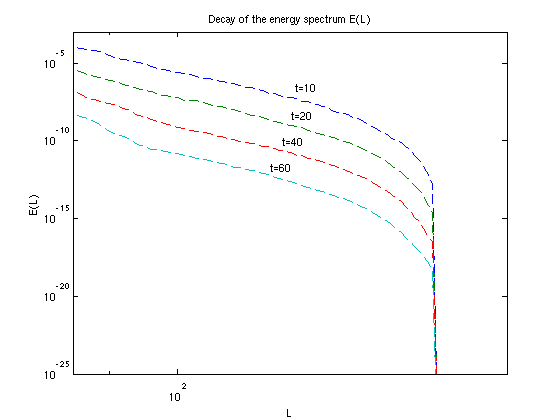}\\
Energy spectrum $E(L)$ of  $\vU_{75}(t)$.  &  Energy spectrum $E(L)$ of 
$\tilde{\Phi}_{150}(\vU_{75}(t))$.
\end{tabular}
\caption{Energy spectra of velocity $\vU_{75}(t)$ 
and $\Phi_{150}(\vU_{75}(t))$.}\label{energy_spectrum}
\end{figure}

\begin{figure}[h]
\begin{tabular}{cc}
\includegraphics[width=7.5cm,height=7.5cm]{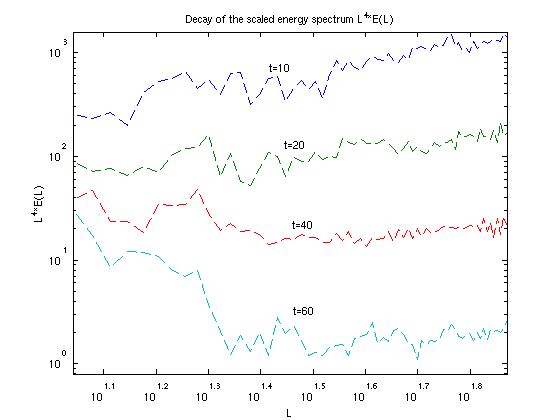}
&
\includegraphics[width=7.5cm,height=7.5cm]{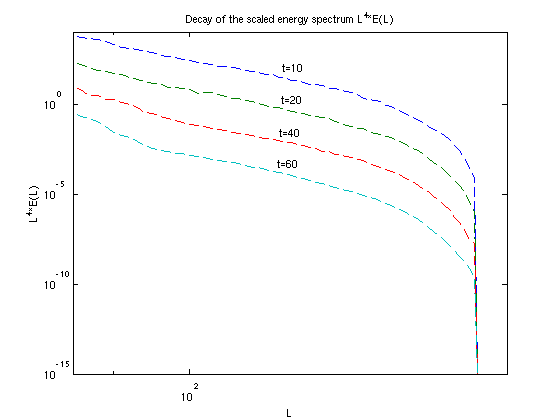}\\
Energy spectrum $L^4*E(L)$ of  $\vU_{75}(t)$ & \hspace{-0.3in}  Energy spectrum
$L^4*E(L)$ of  $\tilde{\Phi}_{150}(\vU_{75}(t))$.
\end{tabular}
\caption{Scaled energy spectra  $L^4*E(L)$  of $\vU_{75}(t)$ 
and $\tilde{\Phi}_{150}(\vU_{75}(t))$.}\label{scaled_energy_spectrum}
\end{figure}
\end{center}

\clearpage
\newpage
\section{Appendix}

In this section, we generalize certain known domain case estimates, that are
fundamental for the NSE analysis, to the  spherical surface case and hence
prove \eqref{Adelta}.
Following~\cite[Page~574]{ilin91}, for  $\vu \in C^\infty(TS)$, 
we extend $\vu$ to the spherical layer
$S \times I$, $I = (r_1,r_2), 0<r_1 < 1 < r_2 < \infty$
by the formula
\be\label{ext_map}
 \tilde{\vu}(\xh) = \varphi(|\xh|)\vu(\xh / |\xh|),
 \ee
 where $\varphi(t) \in C_0^\infty(I)$,   $\varphi(t) \ge 0,~t \in I$, 
and  $\varphi(1) = 1$.
We have
\[
   \int_{S\times I} |\tvu|^p =
        \int_{r_1}^{r_2} \varphi^p(r) \int_{S} |\vu|^p dS
\]
In other words
\be\label{Lpnorm}
  \|\tvu\|_{L^p(S\times I)} = c \|\vu\|_{L^p(TS)},
\ee
where $c=c(\varphi,r_1,r_2)$.
Suppose $\vu,\vv,\vw$ are extended from $S$ to the spherical
layer $S \times I$ by \eqref{ext_map}. Then from
\cite[Lemma 4.3]{ilin91} we have
\be\label{btilde}
    b(\vu,\vv,\vw) = c b(\tilde{\vu},\tilde{\vv},\tilde{\vw}),
\ee

On the sphere $S$, we have the following version of Sobolev embedding 
inequality \cite{aubin}
\be\label{Sobolev_imbedding}
  \| \vu \|_{L^q(TS)} \le C \|\vu\|_{H^{s}(TS)},
       \qquad s<1, \quad \frac{1}{q}=\frac{1}{2} - \frac{s}{2}.
\ee
The following nonlinearity estimate 
is an adaptation of \cite[Proposition 6.1]{constantin_foias}
for $S$. 
\begin{proposition}~\label{foias_prop_S}
Let $s_1,s_2,s_3 \ge 0$ be real numbers, and we assume that
$s_1+s_2+s_3 \ge 1$ and $(s_1,s_2,s_3)\ne (0,0,1),(0,1,0),(1,0,0)$.
Then there exists a constant depending on $s_1,s_2,s_3$ such that
\[
   |b(\vu,\vv,\vw)| \le C \|\vu\|_{H^{s_1}(TS)} \|\vv\|_{H^{s_2+1}(TS)}
                        \| \vw\|_{H^{s_3}(TS)}
\]
or in the extrapolated form, writing $H^{s}(TS) =  H^{s}_{TS}$,   for all $\vu,\vv,\vw \in C^\infty(TS)$
\beas
   |b(\vu,\vv,\vw)| 
   \le C & \|\vu\|^{1+[s_1]-s_1}_{H^{[s_1]}_{TS}} 
           \|\vu\|^{s_1-[s_1]}_{H^{[s_1]+1}_{TS}} 
\|\vv\|^{1+[s_2]-s_2}_{H^{[s_2]+1}_{TS}} 
      \|\vv\|^{s_2-[s_2]}_{H^{[s_2]+2}_{TS}} 
\|\vw\|^{1+[s_3]-s_3}_{H^{[s_3]}_{TS}} 
     \|\vw\|^{s_3-[s_3]}_{H^{[s_3]+1}_{TS}}.
\eeas
\end{proposition}
\begin{proof}
Let $\vu,\vv,\vw \in C^\infty(TS)$ and $\tilde{\vu},\tilde{\vv},\tilde{\vw}$
be their corresponding extension to the spherical layer
$\widetilde{\Omega}:=S\times I$. Let us consider first the case $s_i <1$ for $i=1,2,3$.
Define associated constants  $q_1, q_2, q_2, q_4$ so that  $\sum_{i=1}^4 \frac{1}{q_i} = 1$ and
$\frac{1}{q_i} = \frac{1}{2} - \frac{s_i}{2}$ for $i=1,2,3$.
Then, by H\"older's inequality
\beas
|b(\tvu,\tvv,\tvw)| &=&
\left |\sum_{i,j = 1}^{3}\int_{\widetilde\Omega}
    \tilde{u}_j \frac{\partial \tilde{v}_i}{\partial x_j} \tilde{w}_i 1
     \right| 
    \le \|\tvu\|_{L^{q_1}(\Omega)} \| \nabla \tvv \|_{L^{q_2}(\Omega)}
          \|\tvw\|_{L^{q_3}(\Omega)} \|1\|_{L^{q_4}(\Omega)} \\
\eeas
Restricting to the sphere by \eqref{btilde} and \eqref{Lpnorm} then
using the Sobolev embedding theorem on the sphere \eqref{Sobolev_imbedding}
we have
\beas
   |b(\vu,\vv,\vw)| &\le& C
  \|\vu\|_{L^{q_1}_{TS}} \|\Grad\vv\|_{L^{q_2}_{TS}} \|\vw\|_{L^{q_3}_{TS}} 
    \le
  C \|\vu\|_{H^{s_1}_{TS}} \|\Grad\vv\|_{H^{s_2}_{TS}} \|\vw\|_{H^{s_3}_{TS}}.
\eeas
\end{proof}
\begin{lemma}~\label{lem:weaklip}
Let $\delta \in (1/2,1)$ be given and $\vu, \vv \in V$. Then there
exists $C$, independent of $\vu$ and $\vv$, such that 
\[
  \|\A^{-\delta} \B(\vu,\vv)\| \le C
    \left\{
   \begin{array}{cc}
        \|\A^{1-\delta} \vu \| \|\vv\| & \\
        \|\vu\| \|\A^{1-\delta} \vv\| & \\
   \end{array}, 
   \right. \qquad \vu, \vv \in V.
\]
\end{lemma}
\begin{proof}
Let $\vu,\vv,\vw \in V$.  Hence from \eqref{skew}, and
by using Proposition \ref{foias_prop_S} with $s_1=0$,
$s_2=2\delta-1>0$ and $s_3=2-2\delta>0$, 
\[
   |b(\vu,\vv,\A^{-\delta} \vw)| = |b(\vu,\A^{-\delta}\vw,\vv)|
     \le C \|\vu\| \|\A^{-\delta}\vw\|_{H^{2\delta}(TS)}
         \|\vv\|_{H^{2-2\delta}(TS)}.
\]
Since
$\| \A^{-\delta} \vw\|_{H^{2\delta}(TS)} =
  \| \A^{\delta} \A^{-\delta} \vw\| = \|\vw\|$, we get 
\[
   |b(\vu,\vv,\A^{-\delta} \vw)| \le
   C \|\vu\| \|\vw\| \|\vv\|_{H^{2-2\delta}(TS)}, \qquad  \vw \in V.
\]
Since the inequality is true for all $\vw \in V$, we obtain 
the first bound  
\[
 \|\A^{-\delta} \B(\vu,\vv) \|
    \le C \|\vu\| \|\vv\|_{H^{2-2\delta}(TS)} =
       C \|\vu\| \|\A^{1-\delta}\vv\|.
\]
To obtain the second bound, 
we again use \eqref{skew} and Proposition \ref{foias_prop_S}
but with  $s_1=2-2\delta$, $s_2 = 2\delta-1$ and $s_3=0$, 
\beas
 |b(\vu,\vv,\A^{-\delta}\vw)| =
   |b(\vu,\A^{-\delta}\vw,\vv)|
     &\le& C \| \vu \|_{H^{2-2\delta}(TS)}
             \|\A^{-\delta} \vw\|_{H^{2\delta}(TS)} \|\vv\| \\
     &=& C\|\A^{1-\delta} \vu\| \|\vw\| \|\vv\|,  \qquad  \vw \in V.
\eeas
Since the inequality is true for all $\vw \in H$, we obtain
\[
\|\A^{-\delta} \B(\vu,\vv) \| \le C\|\A^{1-\delta}\vu\| \|\vv\|.
\]
\end{proof}
%
\paragraph{Acknowledgments:} The support of the Australian Research Council
under its Discovery and Centre of Excellence programs is gratefully
acknowledged. The authors thank Professors M. Farge and  E. S. Titi~\cite{Titi2009} for valuable 
discussions.

\end{document}